\documentclass[graybox]{svmult}
\usepackage{mathptmx}       
\usepackage{helvet}         
\usepackage{courier}        
\usepackage{makeidx}         
\usepackage{graphicx}        
\usepackage{multicol}        
\usepackage[bottom]{footmisc}
\usepackage{amsmath,amssymb,amsfonts}        
\usepackage{blkarray,tikz}
\makeindex             

\usepackage{pinlabel}
\usepackage{xspace}
\usepackage{comment}
\usepackage{accents}
\usepackage{caption}
\usepackage{subcaption}
\usepackage{xcolor}%
\usepackage{lineno}

\newcommand{\bH}{\mathbb{H}}
\newcommand{\bG}{\mathbb{G}}

\newcommand{\BR}{\textit{BR}}

\DeclareMathOperator{\ba}{\backslash}
\newcommand{\con}{/}

\newcommand{\R}{T}

\newcommand{\bT}{\mathbb{T}}

\begin{document}

\title*{Constructing a Tutte polynomial for graphs embedded in surfaces}

\author{Iain Moffatt}

\institute{Iain Moffatt \at Department of Mathematics, Royal Holloway, University of London, Egham, TW20 0EX, United~Kingdom. \email{iain.moffatt@rhul.ac.uk}
}

\maketitle

\abstract{
There are several different extensions of the Tutte polynomial to graphs embedded in surfaces. To help frame the different options, here we consider the problem of extending  the Tutte polynomial to cellularly embedded graphs starting from first principles. We offer three different routes to defining such a polynomial and show that they all lead to  the same polynomial. This resulting polynomial is known in the literature under a few different names including the ribbon graph polynomial, and 2-variable Bollob\'as--Riordan polynomial. \\
Our overall aim here  is to use this discussion as a mechanism for providing a gentle introduction to the topic of Tutte polynomials for graphs embedded in surfaces.
}

\section{Introduction}
If you find yourself in need of a version of the Tutte polynomial for \emph{graphs embedded in surfaces}, rather than for graphs in the abstract, then you will be met with a variety of options in the literature, see~\cite{BR01,BR02,GOODALL_2018,Goodall_2020,HUGGETT_2019,Krajewski_2018,KRUSHKAL_2010,Vergnas_1980,NegamiEmb} and the survey~\cite{Chmutov_2022}.
The best known polynomial among these is  the \emph{Bollob\'as--Riordan polynomial}~\cite{BR01,BR02} which is a 4-variable polynomial of embedded graphs. However, most  of the results about these topological Tutte polynomials concern a 2-variable specialisation of this polynomial, which is known variously as the \emph{ribbon graph polynomial}, the \emph{2-variable Bollob\'as--Riordan polynomial}, and the \emph{Tutte polynomial of a cellularly embedded graph}. (Although some important properties, such as connections between tensions and flows, require moving to the other embedded graph polynomials~\cite{GOODALL_2018,Goodall_2020,maya}.) 
While as a community we agree what the Tutte polynomial of a graph is,  there is less clarity around  what the Tutte polynomial of an embedded graph should be.

In this paper we consider the problem of constructing a Tutte polynomial for embedded graphs. We approach this problem by 
going back to first principles and considering how to construct a polynomial for graphs in surfaces that, just as  the classical Tutte polynomial, has a recursive deletion-contraction definition with a base case given by edgeless graphs. 
We detail three approaches for doing this: an approach  through the dichromatic polynomial, a state-sum approach, and  an approach based upon the deletion-contraction cases of the classical Tutte polynomial. All three approaches result in the same polynomial. 

Our discussion here builds upon the work on topological Tutte polynomials in~\cite{HUGGETT_2019}, which in turn built upon work on Tutte polynomials for Hopf algebras in~\cite{Krajewski_2018}. The results in Section~\ref{sec:ap2} and~\ref{sec:ap3} are developed directly from these sources, but the approach detailed in Sections~\ref{sec:di} and~\ref{sec:ap1} is new to this paper. We also give a quasi-tree expansion for our  polynomial in Section~\ref{sec:qt} and give an overview of its properties in Section~\ref{sec:further}.   

\medskip

This paper was developed from a talk given and prepared at the MATRIX \emph{Workshop on Uniqueness and Discernment in Graph Polynomials},  16th -- 27th October 2023. I would like to thank MATRIX for providing an excellent research environment.

\section{Background on embedded graphs}
This section provides an overview of cellularly embedded graphs, ribbon graphs, and some of their standard constructions and terminology. We refer the reader to ~\cite{zbMATH05202336} for additional background on  graph theory, ~\cite{Ellis_Monaghan_2013} for background on embedded graphs and ribbon graphs as we use them here,~\cite{zbMATH04006288} for standard background on embedded graphs, and~\cite{HUGGETT_2019} for further discussion of deletion and contraction for embedded graphs. 

We assume a familiarity with surfaces and the classification of surfaces, but remind the reader of a few relevant facts here. Additional background can be found in many standard topology or topological graph theory texts, for example~\cite[Chapter~1]{zbMATH03716424}. 
We use the term \emph{surface} here to mean a compact Hausdorff topological space in which every point has a neighbourhood homeomorphic to an open 2-dimensional disc. A \emph{surface with boundary}  here means a compact Hausdorff topological space in which every point has a neighbourhood homeomorphic to either a 2-dimensional disc or the upper half-plane.  Its \emph{boundary} is the set of all points $p$ that have a neighbourhood homeomorphic to the half-plane with the homeomorphism mapping $p$ to the origin. 
 Note that here surfaces and surfaces with boundary need not be connected.
 A surface or surface with boundary is \emph{nonorientable} if it contains a subset homeomorphic to a M\"obius band, and is \emph{orientable} otherwise. 
 A connected orientable surface is homeomorphic to a sphere or the connected sum of tori, with its \emph{genus} being the number of tori.   A connected nonorientable surface is homeomorphic to the connected sum of projective planes, with its \emph{genus} being the number of projective planes.  
Up to homeomorphism a connected surface is uniquely determined by its genus and orientability. 
A connected surface with boundary $\Sigma$ consists of a connected surface $\Sigma'$ with the interiors of $k$ discs removed, with $k$ giving its number of boundary components. The genus of  $\Sigma$ is defined as the genus of $\Sigma'$. 
Up to homeomorphism a connected surface with boundary  is uniquely determined by its genus, orientability, and number of boundary components. 
The genus of a non-connected surface with or without boundary is the sum of genus of each of its connected components.

\subsection{Graphs cellularly embedded in surfaces}\label{ss.gins}
We say $\bG=(V,E)$ is a \emph{graph embedded in a surface} $\Sigma$ (or an \emph{embedded graph} for brevity)  if it consists of a set $V$ of points  on $\Sigma$, called \emph{vertices}, and another set $E$ of simple paths in $\Sigma$, called \emph{edges}. The ends of the edges must lie on vertices,  only the ends of edges may meet vertices, and two edges may only intersect where they share a vertex.  Note that we allow the possibility that both ends of an edge meet the same vertex.  We do not require $\Sigma$ to be connected, but we do insist that there is at least one vertex in each connected component of the surface. Note that  the embedded graph $\bG$ has an \emph{underlying graph} $G$ with vertex set $V$, edge set $E$ and incidences determined by the paths in the obvious way. Abusing notation slightly, we say that $\bG$ is an \emph{embedding} of its underlying graph $G$. 
Two embedded graphs are \emph{equivalent} if there is a homeomorphism between the surfaces they are embedded in that induces an isomorphism of the underlying graphs. If the surfaces are orientable, the homeomorphism must be orientation preserving.  We consider embedded graphs up to equivalence.

Deleting the subset of $\Sigma$ formed by the vertices and edges of $\bG$ results in a collection of surfaces with boundary called the  \emph{regions} of $\bG$. (Visualise using a pair of scissors to cut the surface up by snipping along the graph. The result of this gives the regions.) If each region is homeomorphic to a disc then the regions  are called \emph{faces} and  $\bG$ is said to be \emph{cellularly embedded}.  For example, Figure~\ref{embg1.1} shows a graph that is cellularly embedded in a torus, while Figure~\ref{embg1.2} shows a graph that is not cellularly embedded in a torus (one region is a disc but the other is an annulus). Here we are principally interested in cellularly embedded graphs.

\begin{figure}
     \centering
        \begin{subfigure}[c]{0.45\textwidth}
        \centering
        \includegraphics[scale=0.8]{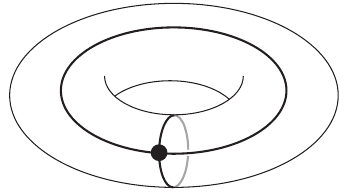}
        \caption{A graph that is cellularly embedded in a torus.}
        \label{embg1.1}
     \end{subfigure}
        \hfill
        \begin{subfigure}[c]{0.45\textwidth}
        \centering
        \includegraphics[scale=0.8]{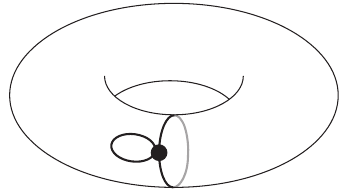}
        \caption{A graph that is not cellularly embedded in a torus.}
        \label{embg1.2}
     \end{subfigure}
\caption{Embedded graphs.}
\label{embg1}
\end{figure}

\medskip 

Defining deletion and contraction of edges
of a cellularly embedded graph requires a little care as we need to ensure that these operations result in a graph that is also cellularly embedded.
Let $\bG$ be a graph cellularly embedded in a surface $\Sigma$, and $e$ be an edge of $\bG$. We would like to define $\bG\ba e$ in the natural way to be the result of removing the edge $e$ from the graph (but leaving the surface unchanged). The difficulty is that this may result in a graph which is not cellularly embedded. (Compare Figure~\ref{delr1} which shows a situation where removing an edge preserves cellular embeddings, and Figures~\ref{del.a2}, and~\ref{del.b2} which shows a situation where it creates non-disc regions.)
In such cases we will modify the surface by replacing any non-disc regions of a non-cellularly embedded graph with discs to obtain a cellular embedding. 
Accordingly we define $\bG$ \emph{delete} $e$, denoted  by $\bG\ba e$, to be the result of the following process. Start by removing $e$ from the set of edges of $\bG$. If this results in a graph that is cellularly embedded in $\Sigma$ then take $\bG\ba e$ to be this cellularly embedded graph. 
Alternatively, if it  results in a graph that is not cellularly embedded in $\Sigma$ then modify the surface using the following process. If removing the edge creates any isolated vertices then place each isolated vertex in its own copy of the sphere. 
 Next delete any non-disc regions from $\Sigma$. Then for each boundary component, take a disc and identify its boundary with the boundary component to obtain a graph cellularly embedded in a new surface $\Sigma'$.
 Take $\bG\ba e$ to be this cellularly embedded graph. See Figure~\ref{del2}.
Note that if $G$ is the underlying graph of $\bG$ then the graph $G\ba e$ (with graph deletion) is the underlying graph of $\bG\ba e$. Also note $\bG$ and $\bG \ba e$ may be  embedded in non-homeomorphic surfaces, see Figure~\ref{dcexamp}.

\begin{figure}[ht]
     \centering
        \hfill
        \begin{subfigure}[c]{0.3\textwidth}
        \centering
        \labellist
        \small\hair 2pt
        \pinlabel $e$ at 66 50
        \endlabellist
        \includegraphics[width=35mm]{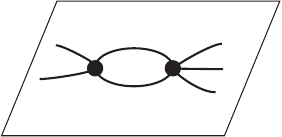}
        \caption{$\bG$.}
        \label{rdel.a1}
     \end{subfigure}
        \hfill
        \begin{subfigure}[c]{0.3\textwidth}
        \centering
        \includegraphics[width=35mm]{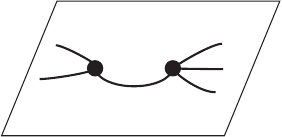}
        \caption{$\bG\ba e$.}
        \label{rcont.b1}
     \end{subfigure}
\caption{An instance of edge deletion.}
\label{delr1}
\end{figure}

\begin{figure}
     \centering
        \hfill
        \begin{subfigure}[c]{0.3\textwidth}
        \centering
        \labellist
\small\hair 2pt
\pinlabel $e$ at 35 28
\endlabellist
\includegraphics[width=35mm]{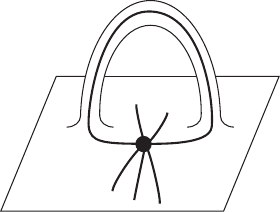}
        \caption{$\bG$ in $\Sigma$.}
        \label{del.a2}
     \end{subfigure}
        \hfill
        \begin{subfigure}[c]{0.3\textwidth}
        \centering
        \includegraphics[width=35mm]{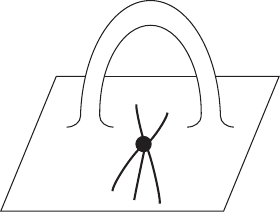}
        \caption{Removing $e$.}
        \label{del.b2}
     \end{subfigure}
        \hfill
        \begin{subfigure}[c]{0.35\textwidth}
        \centering
        \includegraphics[width=35mm]{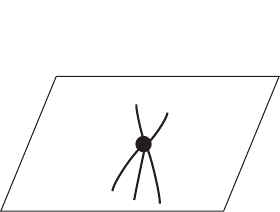}
        \caption{Removing the handle to get $\bG\ba e$.}
        \label{del.c2}
     \end{subfigure}
\caption{An instance where edge deletion involves removing redundant handles.}
\label{del2}
\end{figure}

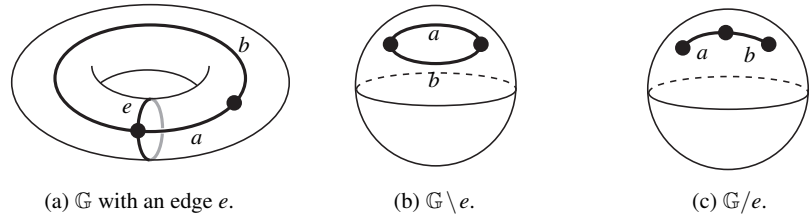
\begin{figure}
     \centering
        \hfill
        \begin{subfigure}[c]{0.3\textwidth}
        \centering
        \labellist
        \small\hair 2pt
        \pinlabel $e$ at 38 21
        \pinlabel $a$ at 60 11
        \pinlabel $b$ at 75 41
        \endlabellist
                \includegraphics[scale=1.2]{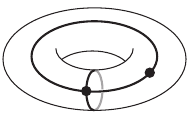}
        \caption{$\bG$ with an edge $e$.}
        \label{dcexamp.1}
     \end{subfigure}
        \hfill
        \begin{subfigure}[c]{0.3\textwidth}
        \centering
           \labellist
        \small\hair 2pt
          \pinlabel $a$ at 29 44
        \pinlabel $b$ at 29 30
        
        \endlabellist
        \includegraphics[scale=1.2]{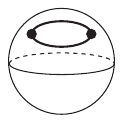}
        \caption{$\bG\ba e$.}
        \label{dcexamp.2}
     \end{subfigure}
        \hfill
        \begin{subfigure}[c]{0.3\textwidth}
        \centering
          \labellist
        \small\hair 2pt
       \pinlabel $a$ at 20 38
        \pinlabel $b$ at 35 38
        \endlabellist
        \includegraphics[scale=1.2]{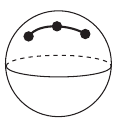}
        \caption{$\bG\con e$.}
        \label{dcexamp.3}
     \end{subfigure}
\caption{An example where deletion and contraction changes the surface.}
\label{dcexamp}
\end{figure}

Contraction leads to another dichotomy. We would like to define $\bG/e$ as the image of the graph $\bG$ in the surface $\Sigma$ under the formation of the topological quotient $\Sigma/e$ which identifies the path $e$ to a point, this point being a new vertex. If we start with a graph cellularly embedded in a surface then sometimes we obtain another graph embedded in a surface, as in  Figure~\ref{conr1}, but sometimes a pinch point is created, as in Figures~\ref{cont.a2} and~\ref{cont.b2}. (the pinch points are points that have neighbourhoods homeomorphic to a cone.) 
In such a situation we  `resolve' pinch points as in Figure~\ref{cont.c2} by `cutting them open', which splits the vertex into two.
More formally, by \emph{resolving a pinch point} we mean the result of the following process. Delete a small neighbourhood of the pinch point. This creates a surface with boundary. Next, by forming the topological quotient space, shrink each boundary component to a point and make this point a vertex. 
Thus we  define $\bG$ \emph{contract} $e$, denoted by $\bG\con e$, to be the cellularly embedded graph obtained by identifying the edge $e$ to a point, then resolving any pinch point that is created.
Observe that if $G$ is the underlying graph of $\bG$, then $G/e$ (with graph contraction) may not be the underlying graph of $\bG/e$. Also $\bG$ and $\bG/ e$ may be graphs embedded in non-homeomorphic surfaces, see Figure~\ref{dcexamp}.

\begin{figure}[ht]
     \centering
        \hfill
        \begin{subfigure}[c]{0.3\textwidth}
        \centering
        \labellist
        \small\hair 2pt
        \pinlabel $e$ at 66 50
        \endlabellist
        \includegraphics[width=35mm]{figs/f1}
        \caption{$\bG$.}
        \label{r2del.a1}
     \end{subfigure}
        \hfill
        \begin{subfigure}[c]{0.3\textwidth}
        \centering
        \includegraphics[width=35mm]{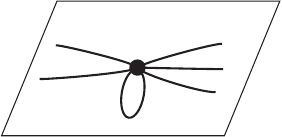}
        \caption{$\bG/e$.}
        \label{r2cont.c1}
     \end{subfigure}
\caption{An instance of edge contraction.}
\label{conr1}
\end{figure}

\begin{figure}
     \centering
        \begin{subfigure}[c]{0.3\textwidth}
        \centering
        \labellist
\small\hair 2pt
\pinlabel $e$ at 107 27
\endlabellist
\includegraphics[height=13mm]{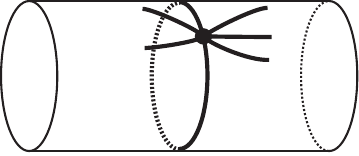}
        \caption{$\bG$ in $\Sigma$.}
        \label{cont.a2}
     \end{subfigure}
        \hfill
        \begin{subfigure}[c]{0.3\textwidth}
        \centering
        \includegraphics[height=13mm]{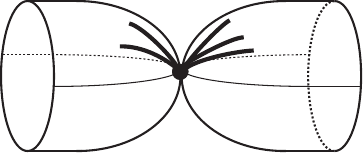}
        \caption{Quotient by $e$.}
        \label{cont.b2}
     \end{subfigure}
\hfill
        \begin{subfigure}[c]{0.3\textwidth}
        \centering
        \includegraphics[height=13mm]{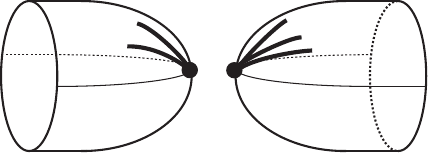}
        \caption{Resolving the pinch point  to get $\bG/ e$.}
        \label{cont.c2}
     \end{subfigure}
\caption{An instance where edge contraction involves resolving pinch points.}
\label{cont2}
\end{figure}

Although not immediately obvious from the above constructions, it can be shown  that for distinct edges $e$ and $f$ of   $\bG$, we have   $(\bG\ba e) \ba f  =  (\bG\ba f) \ba e  $, $(\bG\con e) \con f  =  (\bG\con f) \con e  $, and $(\bG\ba e) \con f  =  (\bG\con f) \ba e  $.  (An easy way to see this is to describe deletion and contraction in terms of arrow presentations, as described in \cite[Section~4.1]{Ellis_Monaghan_2013}. It is then immediate that deletion and contraction act locally at an edge and hence the  commutativity properties hold.) We may therefore make the following definitions. 
If $A$ is a set of edges of $\bG$ then $\bG\ba A$ is the result of deleting each edge in $A$ from $\bG$ (in any order). Similarly, $\bG\con A$ is the result of contracting each edge in $A$  (in any order). An embedded graph of the form $\bG\ba A$ is said to be a \emph{spanning subgraph} of $\bG$.

\medskip 

Here our interest is in cellularly embedded graphs. Henceforth we shall use the term ``embedded graph'' to mean ``cellularly embedded graph'', and similarly ``graphs embedded in surfaces'' to mean ``graphs cellularly embedded in surfaces''.

\subsection{Ribbon graphs} 

When working with polynomials of graphs embedded in surfaces it is common to describe embedded graphs as ribbon graphs.
A \emph{ribbon graph} $\bG =\left(V,E\right)$ is a surface with boundary, represented as the union of two sets of discs: a set $V$ of \emph{vertices} and a set of \emph{edges} $E$ with the following properties.
\begin{enumerate}
 \item The vertices and edges intersect in disjoint line segments.
 \item Each such line segment lies on the boundary of precisely one vertex and precisely one edge. In particular, no two vertices intersect, and no two edges intersect.
 \item Every edge contains exactly two such line segments.
\end{enumerate}
See Figures~\ref{f.descb} and~\ref{f1} for some pictures of ribbon graphs.

Two ribbon graphs are \emph{equivalent} if there is a homeomorphism (which should be orientation-preserving if the ribbon graphs are orientable) that sends vertices to vertices, and edges to edges. We consider ribbon graphs up to equivalence.
In particular equivalence preserves the vertex-edge structure, adjacency, and cyclic ordering of the edges incident to each vertex.
It is worth emphasising that ribbon graph equivalence demands more than that the underlying surfaces with boundary are homeomorphic.

It is well-known that ribbon graphs are just descriptions of graphs cellularly embedded
in surfaces (see for example~\cite{Ellis_Monaghan_2013}, or \cite{zbMATH04006288} where ribbon graphs are called reduced band decompositions). If $\bG$ is a cellularly embedded graph, then a ribbon
graph representation results from taking a small neighbourhood of
the cellularly embedded graph $\bG$, and deleting its complement. On the other hand, if $\bG$ is a
ribbon graph, then, topologically, it is a surface with boundary. Capping off the holes, that is, `filling in' each hole by identifying its boundary component with the boundary of a disc, results in a ribbon graph embedded in a  surface from which a graph embedded in the surface is readily obtained. Figure~\ref{f.desc} shows a cellularly embedded graph and its corresponding ribbon graph.
Two ribbon graphs are equivalent if and only if they describe equivalent cellularly embedded graphs.

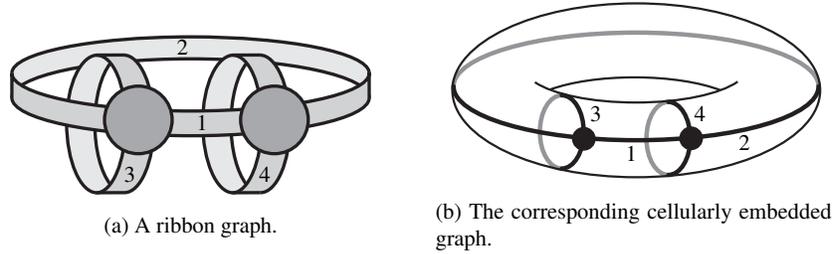
\begin{figure}
     \centering
        \hfill
        \begin{subfigure}[c]{0.45\textwidth}
        \centering
        \labellist \small\hair 2pt
\pinlabel {$1$} at   60 22.7
\pinlabel {$2$} at   54 45.6
\pinlabel {$3$} at   38 7
\pinlabel {$4$} at   79 7
\endlabellist
\includegraphics[height=2.5cm]{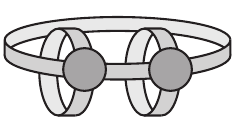}
        \caption{A ribbon graph.}
        \label{f.descb}
     \end{subfigure}
        \hfill
        \begin{subfigure}[c]{0.45\textwidth}
        \centering
        \labellist \small\hair 2pt
\pinlabel {$1$} at   84 14
\pinlabel {$2$} at    135 19
\pinlabel {$3$} at    68 32
\pinlabel {$4$} at   115 32
\endlabellist
\raisebox{0mm}{\includegraphics[height=2.5cm]{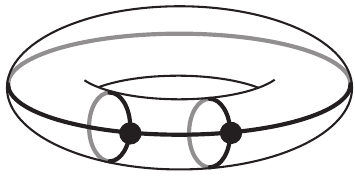}}
        \caption{The corresponding cellularly embedded graph.}
        \label{f.desca}
     \end{subfigure}
\caption{Embedded graphs and ribbon graphs.}
\label{f.desc}
\end{figure}

If $A\subseteq E$, then $\bG\ba A$ is the \emph{spanning ribbon subgraph} of $G=(V,E)$ obtained by \emph{deleting} the edges in $A$. We shall write $\bG\ba e $ for $\bG\ba \{e\} $. 

The definition of edge contraction, introduced in~\cite{BR02,zbMATH05569114}, is a little more involved than that of edge deletion. 
Let $e$ be an edge of $\bG$ and $u$ and $v$ be its incident vertices, which are not necessarily distinct. Then $\bG/e$ denotes the ribbon graph obtained as follows: consider the boundary component(s) of $e\cup u \cup v$ as curves on $\bG$. For each resulting curve, attach a disc, which will form a vertex of $\bG/e$, by identifying its boundary component with the curve. Delete the interiors of $e$, $u$ and $v$ from the resulting complex.
We say that $\bG/e$ is obtained from $\bG$ by \emph{contracting} $e$. If $A\subseteq E$, $\bG/A$ denotes the result of contracting all of the edges in $A$. (The order in which they are contracted does not matter. Indeed distinct edges can be deleted or contracted in any order, a fact that can be verified using the same justification described for embedded graphs at the end of Section~\ref{ss.gins}.)

The local effect of contracting an edge of a ribbon graph is shown in Table~\ref{tablecontractrg}. An  example is given in Figure~\ref{f1}.
Note that contracting an edge in $\bG$ may change the number of vertices, number of components, or orientability.
Deletion and contraction for ribbon graphs are compatible with the deletion and contraction operations for cellularly embedded graphs described above.  

\begin{table}[t]
\centering
\begin{tabular}{|c||c|c|c|}\hline
 &  non-loop & non-orientable loop & orientable loop\\ \hline
\raisebox{6mm}{$\bG$} &
\includegraphics[scale=.25]{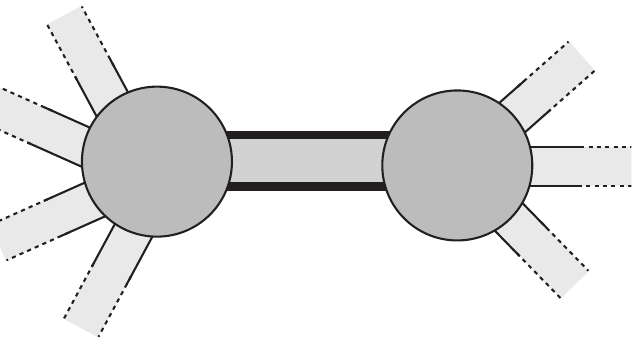} &\includegraphics[scale=.25]{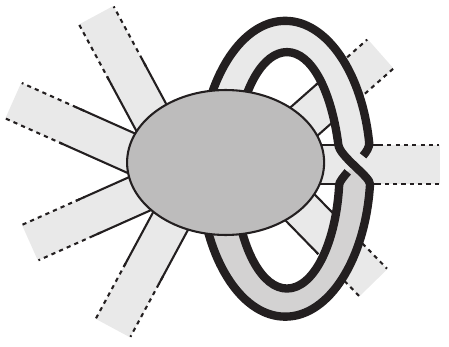} &\includegraphics[scale=.25]{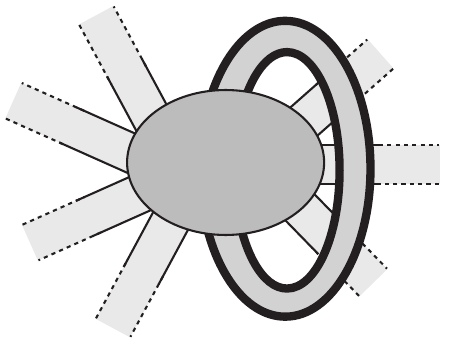}
\\ \hline
\raisebox{6mm}{$G\ba e$}
&
\includegraphics[scale=.25]{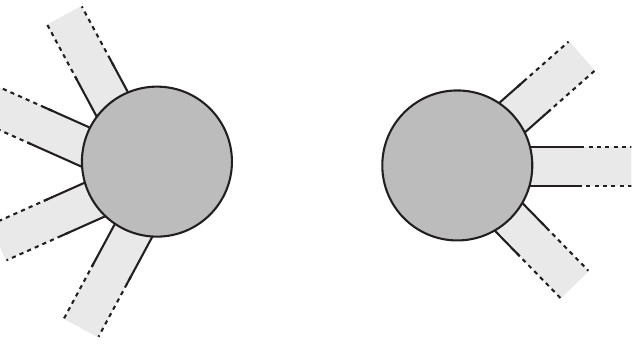} &\includegraphics[scale=.25]{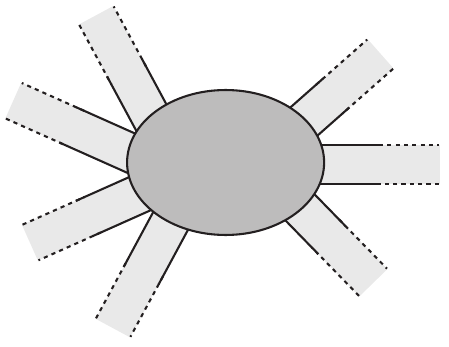}&\includegraphics[scale=.25]{figs/ch4_35a}
\\ \hline
\raisebox{6mm}{\begin{tabular}{l} $\bG/e$ \end{tabular}}
&
\includegraphics[scale=.25]{figs/ch4_35a} &\includegraphics[scale=.25]{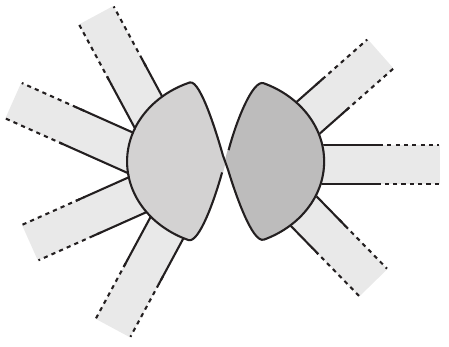}&\includegraphics[scale=.25]{figs/ch4_38a} 
\\ \hline
\end{tabular}
\caption{Operations on an edge $e$ (highlighted in bold) of a ribbon graph.}
\label{tablecontractrg}
\end{table}

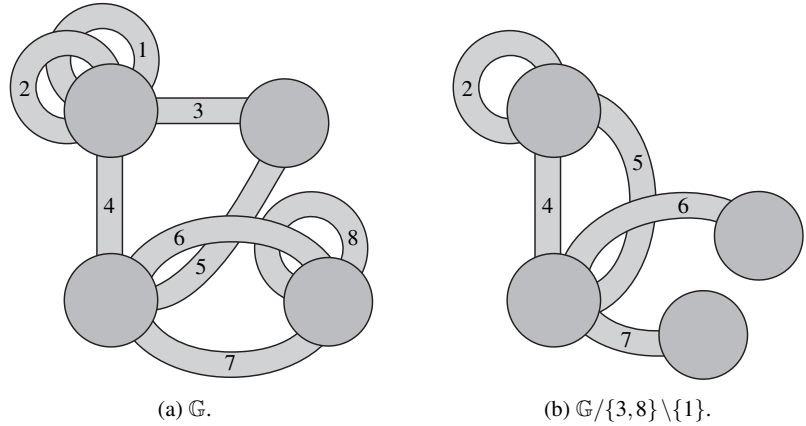
\begin{figure}
     \centering

        \hfill
        \begin{subfigure}[c]{0.45\textwidth}
        \centering
        \labellist
\small\hair 2pt
\pinlabel {$1$}  at  106 228
\pinlabel {$2$}   at    25 201
\pinlabel {$3$}  at    144 186
\pinlabel {$4$}   at   83 121
\pinlabel {$5$}   at     146 79
\pinlabel {$6$}  at   131 99
\pinlabel {$7$}   at    166 13
\pinlabel {$8$}   at  250 99
\endlabellist
\includegraphics[scale=.55]{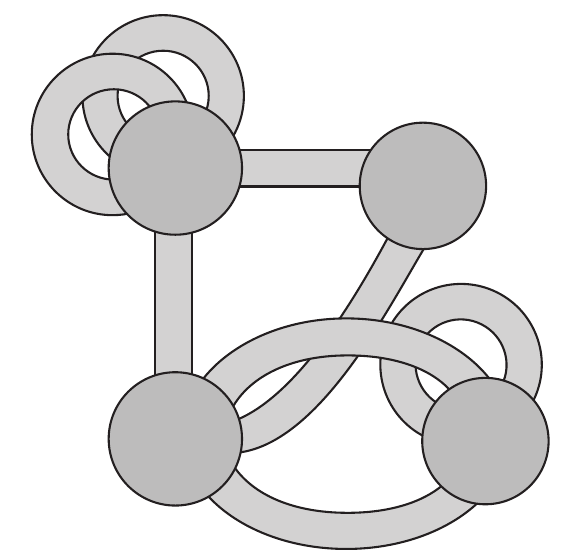}
        \caption{$\bG$.}
        \label{f1a}
     \end{subfigure}
        \hfill
        \begin{subfigure}[c]{0.45\textwidth}
        \centering
        \labellist
\small\hair 2pt
\pinlabel {$2$}   at    25 201
\pinlabel {$5$}  at    142 150
\pinlabel {$4$}   at   80 121
\pinlabel {$6$}  at   173 121
\pinlabel {$7$}   at    134 29
\endlabellist
\includegraphics[scale=.55]{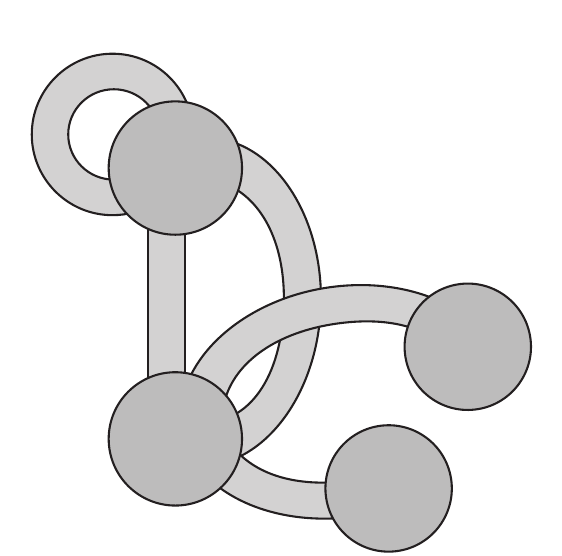}
        \caption{$\bG/{\{3,8\}}  \ba \{1\} $.}
        \label{f1b}
     \end{subfigure}
     \caption{An illustration of ribbon graph operations.}
\label{f1}
\end{figure}

\subsection{Further notation and terminology}

For an embedded graph or ribbon graph $\bG=(V,E)$, and for $A\subseteq E$ we use the following notation and terminology.
\begin{itemize}
\item We use $V(\bG)$ to denote the vertex set of $\bG$, and  $E(\bG)$ to denote its edge set. When it is clear from context we shall often just write $V$ for $V(\bG)$, and $E$ for $E(\bG)$.
\item $e(\bG)$ denotes the number of edges in $\bG$. $e(A)$ denotes $e(\bG \ba (E - A)) =|A|$.
\item $v(\bG)$ denotes the number of vertices in $\bG$. $v(A)$ denotes $v(\bG \ba (E - A)) =v(\bG)$.
\item $k(\bG)$ denotes the number of  components of the underlying graph when $\bG$ is an embedded graph; or, if $\bG$ is a ribbon graph, its number of connected components. $k(A)$ denotes $k(\bG \ba (E - A))$. 
\item $r(\bG) := v(\bG)-k(\bG)$ is the \emph{rank} of $\bG$. $r(A)$ denotes $r(\bG \ba (E - A))$ which equals $v(A)-k(A)$.
\item $b(A)$ denotes the number of boundary components of a small neighbourhood of the subcomplex $V\cup A$ in the graph $\bG=(V,E)$ embedded in a surface $\Sigma$. Observe that $b(E)$ equals the number of faces of $\bG$.  Alternatively, if $\bG=(V,E)$ is realised as a ribbon graph, then   $b(A)$ is the number of boundary components of its ribbon subgraph $(V,A)$. Additionally, $b(\bG)$ denotes $b(E)$.
\item  $\gamma(\bG)$ is the \emph{Euler genus} of $\bG$. It is defined by \emph{Euler's formula}  
\begin{equation}\label{eq:eulers} 
v(\bG)-e(\bG)+b(\bG) = 2k(\bG) - \gamma(\bG). 
\end{equation}
$\gamma(A)$ denotes  $\gamma(\bG \ba (E - A))$. Note that Euler genus is additive over connected components.
\item A graph embedded in a surface is \emph{orientable} if it is embedded in an orientable surface, otherwise it is \emph{nonorientable}. A ribbon graph is \emph{orientable} if it is orientable when considered as a surface with boundary, otherwise it is \emph{nonorientable}. Note that a nonorientable embedded graph or ribbon graph may have orientable components.
\item $g(\bG)$ denotes the \emph{genus} of $\bG$. If $\bG$ is a graph  embedded in a surface this is the genus of the surface it is embedded in. If $\bG$ is a ribbon graph this is its genus as a surface with boundary. 
$g(A)$ denotes  $g(\bG \ba (E - A))$.  If $\bG$ is orientable then $\gamma(\bG)=2g(\bG)$, and if it is nonorientable $\gamma(\bG)=g(\bG)$. Note that  genus is additive over connected components. Finally, $\bG$ is said to be \emph{plane} if $g(\bG)=0$. 

\end{itemize}
We apply the above notation to (non-embedded) graphs where it makes sense to do so.

As an example,  if $\bG$ is as in Figure~\ref{f.desc} then $b(\bG)=2$, $b(\{3,4\})=4$, $\gamma(\bG)=2$, and, as $\bG$  is orientable, it has genus  1.

\section{A Tutte polynomial for graphs embedded in surfaces}\label{sec:clas}
In this section we consider the problem of defining a Tutte polynomial for embedded graphs. We detail three approaches for defining it and show that all three approaches result in the same polynomial. We then give a summary of   properties of this polynomial.

\subsection{A reminder of the classical Tutte polynomial of a graph}
Our staring point is the classical Tutte polynomial of a graph. We refer the reader to~\cite{zbMATH07553843} for additional background on the Tutte polynomial. The \emph{Tutte polynomial}, $T(G;x,y) \in\mathbb{Z}[x,y]$, of a graph $G=(V,E)$ 
is the polynomial uniquely defined by the following recursive  \emph{deletion-contraction relations}:
\begin{equation}\label{eq:tdc}
T(G;x,y) :=\begin{cases}   x\, T(G / e;x,y) &  \text{if $e$ is a bridge,} \\
 y\, T(G\backslash e;x,y ) &  \text{if $e$ is a loop,} \\
 T(G\backslash e;x,y ) + T(G/ e;x,y ) &  \text{if $e$ is neither a bridge or loop,}\\
   1 &  \text{if $G$ is edgeless}.
 \end{cases} \end{equation}
Thus $T(G;x,y)$ is computed by applying the first three relations in~\eqref{eq:tdc} to express $T(G;x,y)$ as a $\mathbb{Z}[x,y]$-linear combination of Tutte polynomials of edgeless graphs, each of which evaluates to 1 by the fourth relation in~\eqref{eq:tdc}. It is not immediate that $T(G;x,y)$, as we have defined it, is independent of the choice  of the order of the edges in which  we apply the relations in~\eqref{eq:tdc}. The standard way of proving that it is indeed independent of this choice is to show that  
\begin{equation}\label{eq:tsum}
T(G;x,y) = \sum_{A\subseteq E} (x-1)^{r(E)-r(A)} (y-1)^{|A|-r(A)},
\end{equation}
where $r(A)$ denotes the rank of the spanning subgraph $(V,A)$ of $G=(V,E)$. The right-hand side of~\eqref{eq:tsum} is clearly well-defined and hence so is $T(G;x,y)$ (details can be found in~\cite{zbMATH01179517}).

\subsection{Defining a polynomial through deletion-contraction relations}\label{sec:di}

We are interested in finding an analogue of Equation~\eqref{eq:tdc}, and hence a version of the Tutte polynomial, for graphs embedded in  surfaces. 
We start by trying to make the most obvious changes to Equation~\eqref{eq:tdc}. First we replace graphs with embedded graphs (so we replace each $G$ in the equation with $\bG$), then our deletion and contraction operations ``$\ba e$'' and ``$\con e$'' become the embedded graph versions, rather than the graph versions.  In the graphs case, Equation~\eqref{eq:tdc} is split into cases according to whether $e$ is a bridge, loop or neither.  It does not make sense to consider these three cases in our topological version.\footnote{It is not hard to see that using the cases in Equation~\eqref{eq:tdc}  for the topological polynomial results in an embedded graph polynomial that equals the Tutte polynomial of the underlying graph.} 
So far we have that our deletion-contraction relations for embedded graphs should be split into cases, but it is not at all obvious what these cases should be. 
Our  approach is to avoid the cases in~Equation~\eqref{eq:tdc} altogether, starting by rewriting the Tutte polynomial of a graph so that it treats every  edge in exactly the same way.

\medskip

The \emph{dichromatic polynomial} $Z(G;u,v) \in \mathbb{Z}[u,v]$ of a graph $G=(V,E)$ can be defined as
\[ Z(G;u,v) := \sum_{A\subseteq E} u^{|A|}v^{k(A)}. \]
By expanding exponents in Equation~\eqref{eq:tsum}, it can be shown that the dichromatic polynomial and Tutte polynomial are equivalent, with
\begin{equation}\label{eq:tzgraphs}
  T(G;x+1,y+1)= x^{-k(G)}y^{-v(G)} Z(G; y, xy). 
  \end{equation}
The significance  of $Z(G;u,v)$ for us here is that its deletion-contraction relations treat all edges in the same way: 
\begin{equation}\label{eq:zdc}
Z(G;u,v) =\begin{cases}  
 Z(G\ba e;u,v ) + u\, Z(G\con e;u,v  ) &  \text{if $e$ is an edge,}\\
   v^{v(G)} &  \text{if }E(G)=\emptyset.
 \end{cases} \end{equation}

Thus we may may proceed in the obvious way by defining an analogue of the dichromatic polynomial for graphs  embedded in surfaces by replacing the graphs in Equation~\eqref{eq:zdc} with embedded graphs and setting 
\begin{equation}\label{eq:z1}
Z(\bG;u,v) :=\begin{cases}  
 Z(\bG\ba e;u,v ) +u Z(\bG\con e;u,v  ) &  \text{if $e$ is an edge,}\\
   v^{v(\bG)} &  \text{if }E(G)=\emptyset.
 \end{cases} \end{equation}
It is worth emphasising that Equation~\eqref{eq:zdc} uses graph deletion and contraction whereas Equation~\eqref{eq:z1} uses embedded graph deletion and contraction.  Although easily overlooked due to the similarity in the notation, this difference is crucial here.

\begin{example}\label{examp:z1}
Figure~\ref{fig:zdc} shows two calculations using Equation~\eqref{eq:z1}. Figure~\ref{fig:zdc1}  shows $Z(\bG;u,v)$  where $\bG$ consists of one vertex and two edges  embedded in the torus computed in terms of embedded graphs. This gives $Z(\bG;u,v)=v+2uv^2+u^2v$. 
Figure~\ref{fig:zdc2}  shows $Z(\bG;u,v)$ where $\bG$ consists of a 2-cycle embedded in the real projective plane computed in terms of ribbon graphs. This gives  $Z(\bG;u,v)=v^2+2uv+u^2v$. 
\end{example}

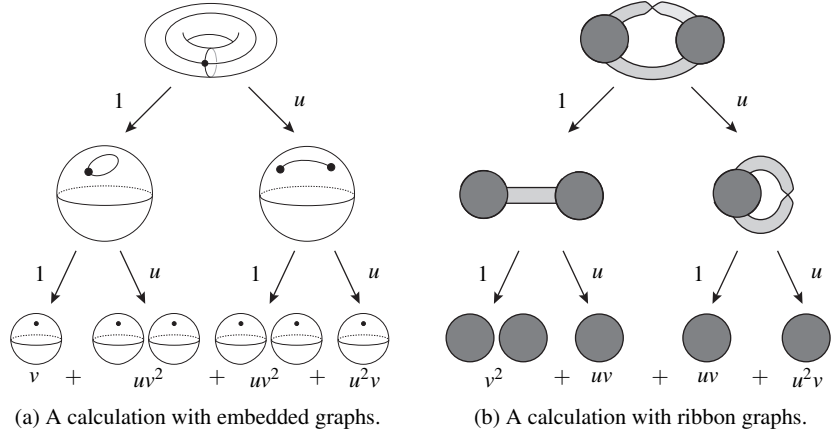
\begin{figure}[ht]
   \centering
        \hfill
        \begin{subfigure}[c]{0.45\textwidth}
        \centering
       \labellist
\small\hair 2pt
\pinlabel $1$ at 180 436
\pinlabel $u$ at 452 435
\pinlabel $1$ at 61 175
\pinlabel $u$ at 234 175
\pinlabel $1$ at 386 175
\pinlabel $u$ at  564 175
\pinlabel $v$ at  52 18
\pinlabel $uv^2$ at  225 18
\pinlabel $uv^2$ at  400 18
\pinlabel $u^2v$ at  550 18
\pinlabel $+$ at  111 18
\pinlabel $+$ at  328 18
\pinlabel $+$ at  479 18
\endlabellist
\includegraphics[scale=0.25]{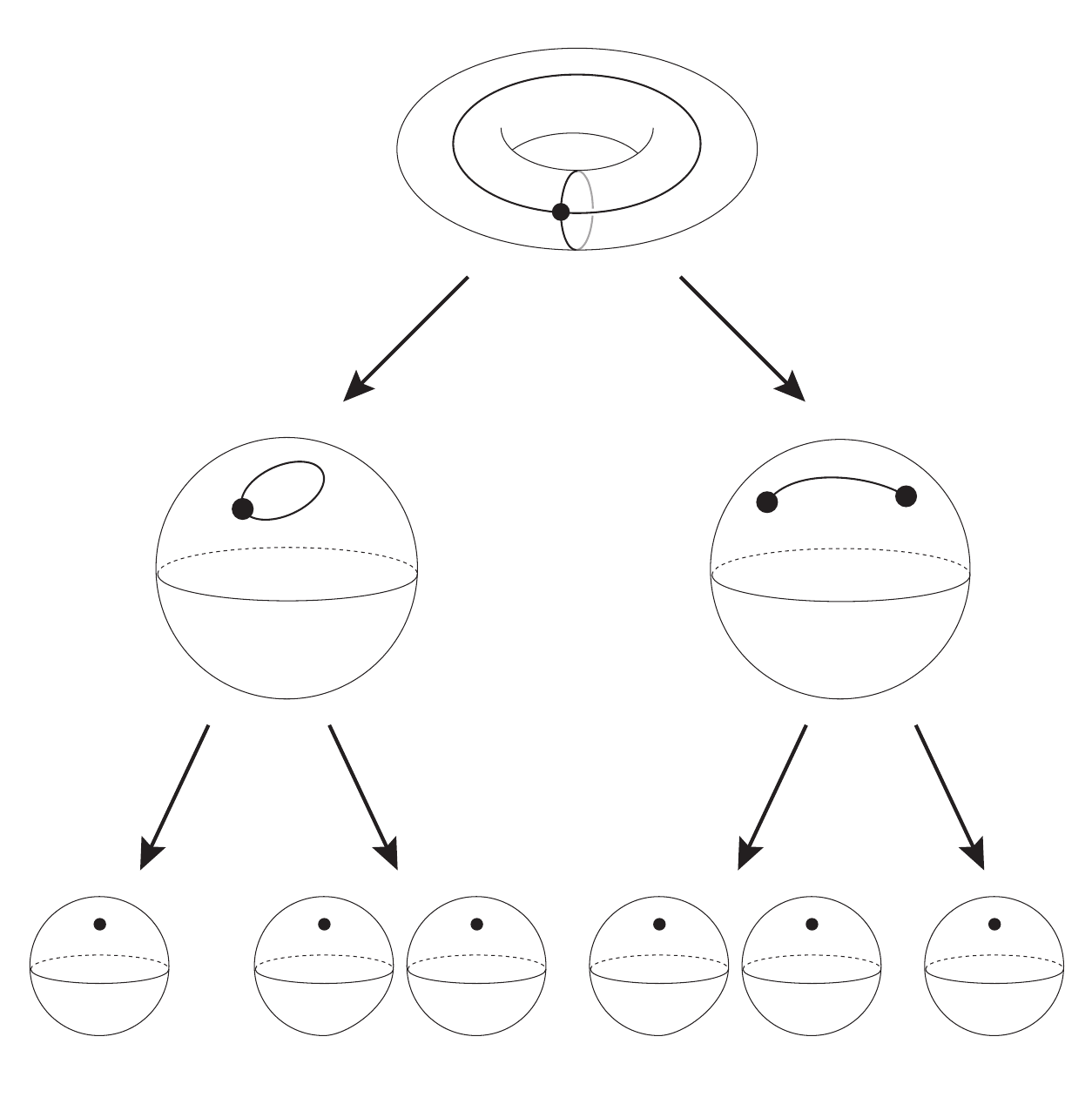}
        \caption{A calculation with embedded graphs.}
        \label{fig:zdc1}
     \end{subfigure}
        \hfill
        \begin{subfigure}[c]{0.45\textwidth}
        \centering
               \labellist
\small\hair 2pt
\pinlabel $1$ at 180 436
\pinlabel $u$ at 452 435
\pinlabel $1$ at 61 175
\pinlabel $u$ at 234 175
\pinlabel $1$ at 386 175
\pinlabel $u$ at  564 175
\pinlabel $v^2$ at  78 18
\pinlabel $uv$ at  240 18
\pinlabel $uv$ at  400 18
\pinlabel $u^2v$ at  550 18
\pinlabel $+$ at  180 18
\pinlabel $+$ at  328 18
\pinlabel $+$ at  479 18
\endlabellist
\includegraphics[scale=0.25]{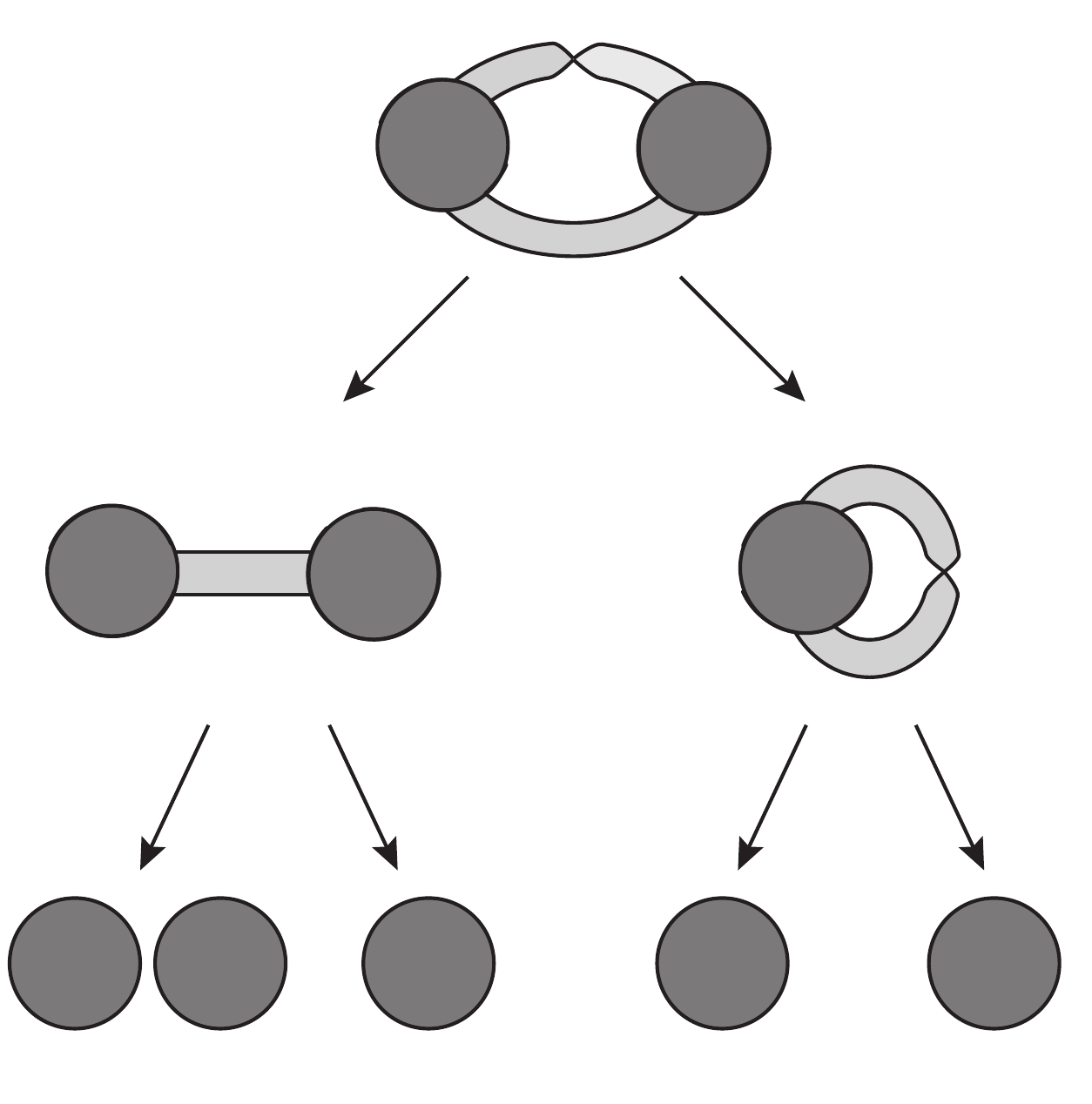}
        \caption{A calculation with ribbon graphs.}
        \label{fig:zdc2}
     \end{subfigure}
     \caption{Two deletion-contraction computations of $Z(\bG;u,v)$.}
\label{fig:zdc}
\end{figure}

It is not immediately obvious from~\eqref{eq:z1}  that $Z(\bG;u,v)$  is well-defined (i.e., that the resulting value is independent of the order that edges are chosen in any calculation). However it is routine to show  that 
\begin{equation}\label{eq:z2}
Z(\bG;u,v) = \sum_{A\subseteq E} u^{|A|}v^{b(A)}, 
\end{equation}
and so $Z(\bG;u,v)$  is indeed well-defined.

\begin{example}\label{examp:z2}
If $\bG$ is the embedded graph or ribbon graph shown in Figure~\ref{f.desc}, then the contribution of each subset $A$ of edges to the sum in Equation~\eqref{eq:z2} is  given in the following table. $Z(\bG;u,v)$ is obtained by summing these contributions.
\begin{center}
{\renewcommand{\arraystretch}{1.2}
\begin{tabular}{|c|c|c|c|c|c|c|c|c|}
\hline
 $\{1,2,3,4\} $	&$\{1,3,4\}$ 	& $\{1,2,3\}$	&$\{1,2\}$	&$\{3,4\}$	&$\{1,3\},\{1,4\}$	&$\{1\}$	&$\{3\}$	&$\emptyset$\\
			& $\{2,3,4\}$	&$\{1,2,4\}$	&	&	&$\{2,3\},\{2,4\}$	&$\{2\}$	&$\{4\}$	&
			\\ \hline
 $u^4v^2$	&2	$u^3v^3$&	$2u^3v$	&$u^2v^2$	&$u^2v^4$	&$4 u^2v^2$	&$2uv$	&$2uv^3$ & $ v^2$
\\\hline
\end{tabular}
}
\end{center}
\end{example}

We observe that if $\bG$ is of genus zero  with underlying graph $G$ (i.e., if $\bG$  is an embedding of $G$ into a disjoint union of spheres) then $Z(\bG;u,v)$ and $Z(G;u,v)$ do not coincide. Instead, by making use of Euler's formula, we see that 
\begin{equation}\label{eq:zzplane}
Z(G;u,v) = \sqrt{v}^{v(\bG)}\; Z(\bG; u/\sqrt{v} , \sqrt{v}).
\end{equation}

\subsection{Approach 1: Deducing a Tutte polynomial from $Z(\bG)$}\label{sec:ap1}

One difficulty of this approach is that we have defined an analogue of the dichromatic polynomial rather than the Tutte polynomial. So what should the Tutte polynomial be? 
We know that the Tutte polynomial is equivalent to the dichromatic polynomial as in Equation~\eqref{eq:tzgraphs}. Thus, by analogy,  we should expect our Tutte polynomial $T(\bG;x,y)$ of embedded graphs to be  obtained from  $Z(\bG;u,v)$ by a change of variables. But which change of variables exactly?

To decide this we need to consider what properties $T(\bG;x,y)$ should have. Here we impose the reasonable condition that when $\bG$ is an embedding of a graph $G$ in the plane or sphere (or more generally when $g(\bG)=0)$) we have
\begin{equation}\label{eq:tt}
T(\bG;x,y) =T (G;x,y),
\end{equation}
i.e., the Tutte polynomial of a graph and of an embedded graph coincide for   graphs embedded in the plane.
With this we have 
  \begin{align*}
  T (G;x+1,y+1)  
  & =  x^{\,-k(G)}y^{\,-v(G)} Z(G; y, xy) \quad\text{(by~\eqref{eq:tzgraphs})}
  \\&=x^{\,-k(G)}y^{\,-v(G)} \sqrt{xy}^{\,v(\bG)}\; Z(\bG; y/\sqrt{xy} , \sqrt{xy}) \quad\text{(by~\eqref{eq:zzplane})}
\\&=\sqrt{x}^{\,e(\bG)-b(\bG)} \sqrt{y}^{\,-v(\bG)}\; Z(\bG; \sqrt{y/x} , \sqrt{xy}) \quad\text{(by~Euler's formula).}
  \end{align*}
It is then natural to propose the following definition.
\begin{definition}[Tutte polynomial --- first definition]\label{def:ts1}
The \emph{Tutte polynomial}, $T(\bG;x,y)$, of  an embedded graph  $\bG$ is defined as
\begin{equation}\label{eq:tdef1}
  T (\bG;x+1,y+1)  := \sqrt{x}^{\,e(\bG)-b(\bG)} \sqrt{y}^{\,-v(\bG)}\; Z(\bG; \sqrt{y/x} , \sqrt{xy}).
\end{equation}
\end{definition}

\begin{example}\label{examp:z3}
Making use of Example~\ref{examp:z1}, we see that when  $\bG$ consists of one vertex and two edges embedded in the torus, 
$T (\bG;x,y) = (x-1)+2(x-1)(y-1)+(y-1)=2xy-x-y$; 
and when $\bG$ consists of a 2-cycle embedded in the real projective plane, 
$T (\bG;x,y) = \sqrt{x-1}^{\,3}+2\sqrt{x-1} +\sqrt{y-1}$.
Notice that the first example is a polynomial in $x$ and $y$, but the second is a polynomial in $\sqrt{x-1}$ and $\sqrt{y-1}$. We shall see later that this is related to  orientability.
\end{example}

Thus in this definition of $T (\bG;x,y)$ we have been naturally led to a graph polynomial that:
\begin{enumerate}
\item coincides with the classical Tutte polynomial on plane graphs; and 
\item can be defined by  deletion-contraction relations in which edgeless graphs give the terminal forms, just as with the classical Tutte polynomial of a graph. (The arising deletion-contraction appears later in Definition~\ref{def3}.)
\end{enumerate}
Although we reached this definition of $T(\bG;x,y)$ by taking the obvious direction at each stage,  the precise form on the right-hand side of Equation~\eqref{eq:tdef1} may appear rather strange. However, we shall shortly make use of Euler's formula to show that it really is what we should expect.

\subsection{Approach 2: A state-sum approach to the Tutte polynomial of an embedded graph}\label{sec:ap2}
In this section we discuss an approach based upon the state-sum expression for the Tutte polynomial of a graph where we modify the rank function so that it records some topological information about an embedded graph. 

\medskip

Recall the state-sum definition of the classical Tutte polynomial of a graph from 
Equation~\eqref{eq:tsum}:
\begin{equation}\label{eq:tsumrep}
T(G;x,y) = \sum_{A\subseteq E} (x-1)^{r(E)-r(A)} (y-1)^{|A|-r(A)},
\end{equation}
where the rank $r(A)$ equals $v(A)-k(A)$.

 The rank $r(\bG)=v(\bG)-k(\bG)$ records no topological information about an embedded graph $\bG$ (in the sense that different embeddings of the same graph will always have the same rank). Hence~\eqref{eq:tsum} does not directly result in a satisfactory topological graph polynomial. Let us modify it so that it does.

We start by setting 
\[\rho(\bG) :=  r(\bG) +\tfrac{1}{2}\gamma(\bG).  \]
Following our standard notational conventions, for a subset of edges $A$, we set $\rho(A) :=  \rho(\bG \ba (E-A)) $ and so
\[
\rho(A) =  r(A) +\tfrac{1}{2}\gamma(A).
\]
We can make use of Euler's formula to write 
\[\rho(A) = \tfrac{1}{2}(|A|+v(A)-b(A)).\]

 With $\rho(A)$  we  define a Tutte polynomial for embedded graphs by replacing $r(A)$ in~\eqref{eq:tsumrep} with $\rho(A)$.

\begin{definition}[Tutte polynomial --- second definition]\label{def2}
The \emph{Tutte polynomial}, $T(\bG;x,y)$, of  an embedded graph  $\bG$ with edge set $E$ is defined as
\begin{equation}\label{eq:tdef2}
  T (\bG;x,y)  := \sum_{A\subseteq E} (x-1)^{\rho(E)-\rho(A)} (y-1)^{|A|-\rho(A)}.
\end{equation}
\end{definition}
 Note that if $\bG$ is an embedding of a graph $G$ in a  disjoint union of spheres,  $\gamma(A)$ is always 0. In this case it follows that $\rho(A)=r(A)$ and therefore   $T (\bG;x,y)$ coincides with the classical Tutte polynomial for plane graphs.

\begin{example}\label{examp:z4}
If $\bG$ is the embedded graph or ribbon graph shown in Figure~\ref{f.desc}, then the contributions of each subset $A$ of edges to the sum in Equation~\eqref{eq:tdef2} are given in the following table. $T(\bG;x,y)$ is obtained by summing these contributions.
\begin{center}
{\renewcommand{\arraystretch}{1.2}
\begin{tabular}{|c|c|c|c|c|}
\hline
 $\{1,2,3,4\} $	&$\{1,3,4\}$ 	& $\{1,2,3\}$	&$\{1,2\}$	&$\{3,4\}$	\\
			& $\{2,3,4\}$	&$\{1,2,4\}$	&		&
			\\ \hline
 $(y-1)^{2}$	&	$2(x-1)(y-1)^{2}$&	$2(y-1)$	&$(x-1)(y-1)$	&$(x-1)^{2}(y-1)^{2}$	
\\\hline
\hline
 $\{1,3\},\{1,4\}$	&$\{1\}$	&$\{3\}$	&$\emptyset$&\\
$\{2,3\},\{2,4\}$	&$\{2\}$	&$\{4\}$	& &
			\\ \hline
 $4(x-1)(y-1)$	&$2(x-1)$	&$2(x-1)^{2}(y-1)$ & $(x-1)^{2}$ &
\\\hline
\end{tabular}
}
\end{center}
\end{example}
Giving $T(\bG;x,y) = x^{2} y^{2}+x y -x -y$.

\medskip 

It is easy to see that $r(E) \geq r(A)\geq 0$ and  $\gamma(E) \geq \gamma(A)\geq 0$ so $\rho(E)-\rho(A)\geq 0$. It is a little harder to see that $ |A|-r(A) \geq \gamma(A)$. (To see this, note that $|A|-r(A)$ is the number of edges that are not in a spanning tree of $\bG\ba (E-A)$. The edges that are  not in a spanning tree generate the first homology group of the surface the graph is embedded in and so cannot be smaller than $\gamma(A)$.) Thus $|A|-\rho(A)\geq 0$. 
Additionally, notice that if $\bG$ is orientable then all embedded subgraphs of it are, giving that $\gamma(A)$ is even and so $\rho(A)$ is integral. 
Combining these observations gives that:
\begin{enumerate}
\item for any embedded graph $\bG$ we have that $T(\bG;x,y)$ is a polynomial in $\sqrt{x-1}$ and $\sqrt{y-1}$; 
\item if  $\bG$ is orientable, $T(\bG;x,y)$ is a polynomial in $x$ and $y$.
\end{enumerate}

In Definitions~\ref{def:ts1} and~\ref{def2} we have seen two different approaches to constructing a Tutte polynomial for embedded graphs. The two approaches result in identical polynomials since, making use of Euler's formula to write $\rho(A) = \tfrac{1}{2}(|A|+v(A)-b(A))$ we have
\begin{align*} 
\sqrt{x}^{\,e(\bG)-b(\bG)} \sqrt{y}^{\,-v(\bG)}\; Z(\bG; \sqrt{y/x} , \sqrt{xy})
&=
\sum_{A\subseteq E} \sqrt{x}^{\,( |E|-b(E)) - ( |A|-b(A))}\sqrt{y}^{\,|A|-|V|+b(A)}
\\&=
\sum_{A\subseteq E} x^{\rho(E) - \rho(A)} y^{|A|-\rho(A)}. 
\end{align*}

\medskip

The connection between $T(\bG;x,y)$ and the well-known polynomial of Bollob\'as and Riordan  is readily seen by considering Definition~\ref{def2}.
The \emph{Bollob\'as--Riordan polynomial}~\cite{BR01,BR02} is \[ \BR(\bG; x,y,z,w):= \sum_{A \subseteq E}   (x-1)^{r( E ) - r( A )}   y^{|A|-r(A)} z^{\gamma(A)} w^{t(A)}, \] 
where $t(A)$ is $0$ when $\bG\ba (E-A)$ is orientable, and is one otherwise; and  $w^2=1$.
By expanding the exponents and collecting terms it is readily verified that 
\[   (x-1)^{\tfrac{1}{2} \gamma(\bG)} \BR\Big(x,y-1, \tfrac{1}{\sqrt{(x-1)(y-1)}},1\Big) =T(\bG;x,y). \]
Because of this $T(\bG;x,y)$ is sometimes called the \emph{2-variable Bollob\'as--Riordan polynomial.}

\subsection{Approach 3: deletion-contraction, duality and the Tutte polynomial of an embedded graph}\label{sec:ap3}

In this section we  approach a definition of the Tutte polynomial for embedded graphs by adapting the deletion-contraction relations for the classical Tutte polynomial of a graph shown in~\eqref{eq:tdc}. The approach given here is extracted from the results of~\cite{HUGGETT_2019}, although that reference comes to the polynomial through a slightly different route and so, along with~\cite{Krajewski_2018}, gives a fourth approach to defining our Tutte polynomial.

\medskip

We need some additional background on geometric duals.
The construction of the \emph{geometric dual} $\bG^*$ of a graph $\bG=(V,E)$ embedded in a surface $\Sigma$ is well known: $V(\bG^*)$ is obtained by placing one vertex in each face of $\bG$, and $E(\bG^*)$ is obtained by embedding an edge of $\bG^*$ between two vertices whenever the faces of $\bG$ in which they lie are adjacent. For example, the geometric dual of a 3-cycle embedded in the sphere is a theta-graph embedded in the sphere.

Geometric duality has a particularly neat description in the language of ribbon graphs. Let $\bG=(V(\bG),E(\bG))$ be a ribbon graph. Recalling that  a ribbon graph is a surface with boundary, we cap off the holes using a set of discs (i.e., take a disc for each boundary component, and identify the boundary component with the boundary of the disc), denoted by $V(\bG^*)$, to obtain a surface without boundary. The \emph{geometric dual} of $\bG$ is the ribbon graph $\bG^*=(V(\bG^*),E(\bG))$. The edges of $\bG$ and $\bG^*$ correspond. 
We shall use $e^*$ to denote the edge of $\bG^*$ that corresponds to an edge $e$ of $\bG$. (Additional examples and discussion can be found in, for example,~\cite{Ellis_Monaghan_2013}.)

  \medskip
  
We initially defined the Tutte polynomial  $T(G;x,y)$ of a  (non-embedded) graph $G$ through deletion-contraction relations in Equation~\eqref{eq:tdc}. The difficulty in defining a Tutte polynomial for embedded graphs directly from these deletion-contraction relations is that it is not at all clear what the cases in the relations should be. However,  rewriting Equation~\eqref{eq:tdc} offers a route to determining the cases.

Let $G$ be a  graph embedded in the sphere (better still, for readers familiar with matroids,  let $G$  be a matroid).
Writing $T(G)$ for $T(G;x,y)$, Equation~\eqref{eq:tdc} may be written as
\begin{equation}\label{eq:tdd1}
T(G) = \begin{cases} p(e) \cdot T(G\ba e) +  q(e) \cdot T(G / e)  & \text{for an edge $e$,} 
\\1 & \text{if $G$ is edgeless.} \end{cases}
\end{equation}
where 
\begin{equation}\label{eq:tdd2}
\begin{array}{l}
p(e) =\begin{cases} x-1 & \text{if $e$ is a loop in $G^*$,} \\ 1 &\text{if $e$ is not a loop in $G^*$; and} \end{cases}
\\ 
\\
q(e) =\begin{cases} y-1 & \text{if $e$ is a loop in $G$,} \\ 1 &\text{if $e$ is not a loop in $G$.}\end{cases}
\end{array}
 \end{equation}
 To see why this is the same as Equation~\eqref{eq:tdc} note that $e$ is a loop in a  graph embedded in the sphere  if and only if $e^*$ is a bridge in the dual. (This is not true for graphs embedded in higher genus surfaces.)

\medskip

To deduce a Tutte polynomial $T(\bG;x,y)$ for embedded graphs from this we can immediately write down an analogue of Equation~\eqref{eq:tdd1}:
\begin{equation}\label{eq:tdd3}
T(\bG) = \begin{cases}  f(e) \cdot  T(\bG\ba e) +  g(e) \cdot  T(\bG / e)  ,  & \text{for an edge $e$,} 
\\1 & \text{if $\bG$ is edgeless.} \end{cases}
\end{equation}
where we write $T(\bG)$ for $T(\bG;x,y)$, and where suitable $f(e)$ and $g(e)$ are yet to be found. 
It remains to find the analogue of~\eqref{eq:tdd2}. 

The key observation here is that there are two different types of loop in an embedded graph. As with graphs an edge of an embedded graph or ribbon graph is a \emph{loop} if it is incident to exactly one vertex. A loop of a graph embedded in a surface is said to be \emph{orientable} if it has a  neighbourhood homeomorphic to an annulus, and it is said to be \emph{nonorientable} if it has a  neighbourhood homeomorphic to an M\"obius band. Correspondingly, a loop in a ribbon graph  is said to be \emph{orientable} if together with its incident vertex it forms an annulus, and it is said to be \emph{nonorientable} if it forms a M\"obius band.  See Table~\ref{tablecontractrg}.

Thus, for some suitable terms $a,b,c,a^*,b^*,c^*$, we set
\begin{equation}\label{eq:tdd4}
\begin{array}{l}
f(e) =\begin{cases} a^* & \text{if $e^*$ is an orientable loop in $\bG^*$,} \\ 
b^* & \text{if $e^*$ is a  nonorientable loop in $\bG^*$,} \\ 
c^* & \text{if $e^*$ is a  not a  loop in $\bG^*$;}  \end{cases}
\\ \text{and}\\
g(e) =\begin{cases} a & \text{if $e$ is an orientable loop in $\bG$,} \\ 
b & \text{if $e$ is a nonorientable loop in $\bG$,} \\ 
c & \text{if $e$ is   not a loop in $\bG$.}  \end{cases}
\end{array}
 \end{equation}

 Let us consider what $a,b,c,a^*,b^*,c^*$ can be. First we would like to $T(\bG;x,y)$  to coincide with $T(G;x,y)$ whenever $\bG$ is  plane. Thus we must have 
 \[a^*=x-1, \quad c^*=1, \quad a=y-1, \quad \text{and} \quad c=1. \]
 For $b^*$ and $b$ we recall that we need the value of $T(\bG;x,y)$ to be independent of the order upon which~\eqref{eq:tdd3} is applied to the edges of an embedded graph.
 Figure~\ref{fwda} shows a ribbon graph $\bG$ and Figure~\ref{fwdb} its dual $\bG^*$. (Note that in this example although  $\bG$ and $\bG^*$ are equivalent, $e$ is the orientable loop in $\bG$, but $e^*$ is the nonorientable loop in $\bG^*$.)   Figure~\ref{fwdc} show the result of computing $T(\bG;x,y)$ by applying~\eqref{eq:tdd3} to the edge $e$ then to the edge $f$.  Figure~\ref{fwdd}  shows the result of applying it to $f$ then to $e$.
 Thus we require that $(b^*)^2+ac=  a^*c^*+b^2$. Recalling our existing values of $a^*,c^*,a,c$ we see that taking $b^*=\sqrt{x-1}$ and  $b=\sqrt{y-1}$ will suffice for this example, and, as we shall presently see, is sufficient to ensure $T(\bG;x,y)$ is well-defined in general.
 
 \begin{figure}
     \centering

        \hfill
        \begin{subfigure}[c]{0.45\textwidth}
        \centering
  \labellist
\small\hair 2pt
\pinlabel $e$ at 19 158
\pinlabel $f$ at 170  117
\endlabellist
\includegraphics[width=30mm]{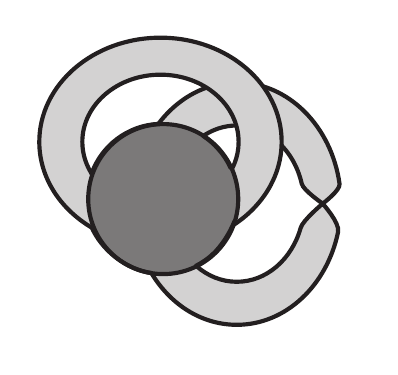}

        \caption{$\bG$.}
        \label{fwda}
     \end{subfigure}
        \hfill
        \begin{subfigure}[c]{0.45\textwidth}
        \centering
  \labellist
\small\hair 2pt
\pinlabel $e^*$ at 19 158
\pinlabel $f^*$ at 170  117
\endlabellist
\includegraphics[width=30mm]{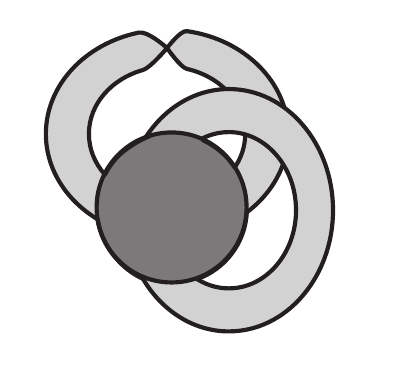}
        \caption{$\bG^*$.}
        \label{fwdb}
     \end{subfigure}
     
             \hfill
        \begin{subfigure}[c]{0.45\textwidth}
        \centering
  \labellist
\small\hair 2pt
\pinlabel $b^*$ at 140 412
\pinlabel $a$ at 375  412
\pinlabel $b^*$ at 44 168
\pinlabel $b$ at 175 168
\pinlabel $a^*$ at 341 168
\pinlabel $c$ at  510 168
\pinlabel $(b^*)^2$ at  43 24
\pinlabel $b^*b$ at 180  24
\pinlabel $a^*a$ at  356 24
\pinlabel $ac$ at  530 24
\pinlabel $+$ at  113 24
\pinlabel $+$ at  247 24
\pinlabel $+$ at  460 24
\endlabellist
\includegraphics[scale=0.25]{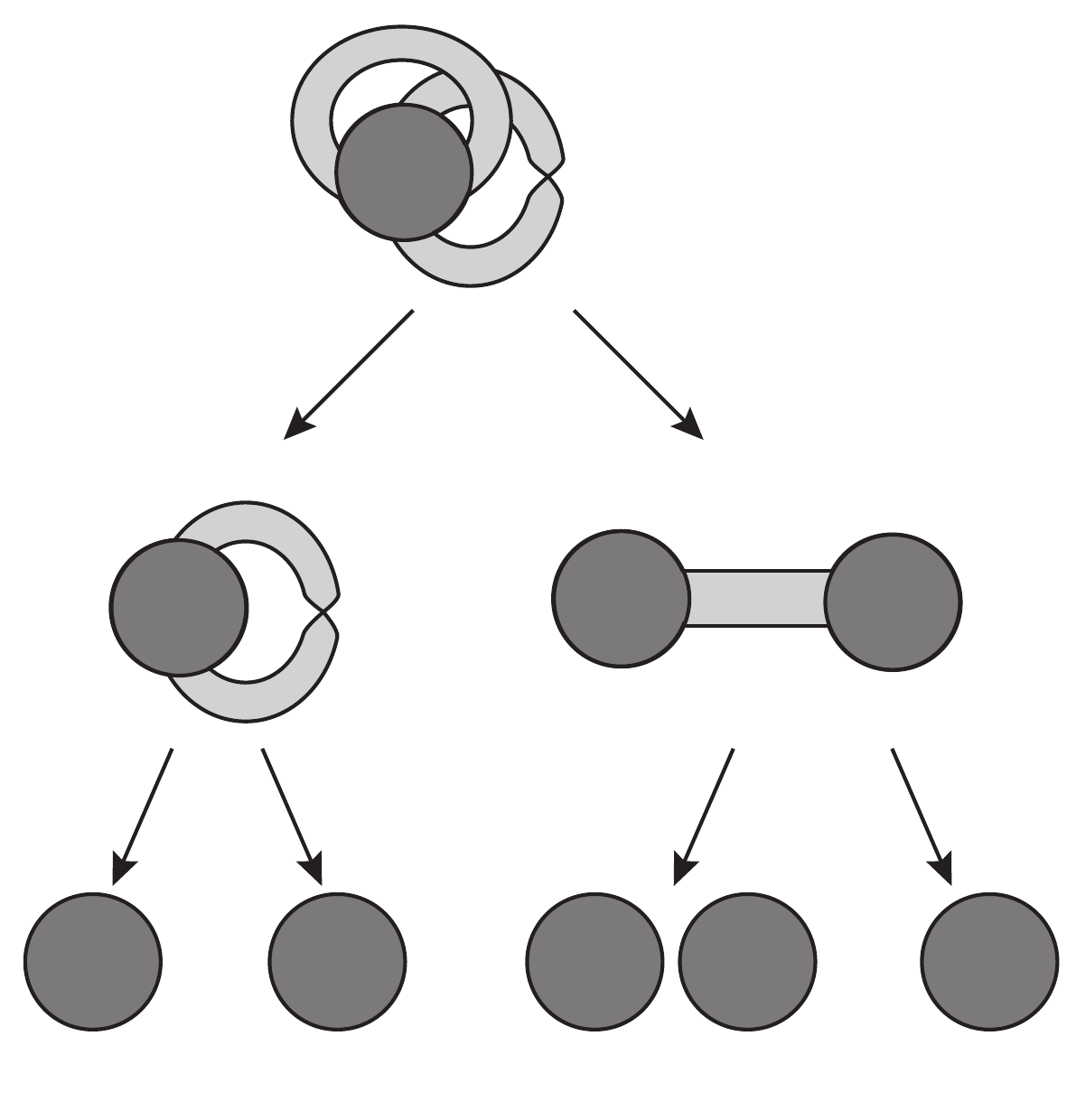}
        \caption{Applying to $e$ then $f$.}
        \label{fwdc}
     \end{subfigure}
        \hfill
        \begin{subfigure}[c]{0.45\textwidth}
        \centering
  \labellist
\small\hair 2pt
\pinlabel $a^*$ at 140 412
\pinlabel $b$ at 375  412
\pinlabel $c^*$ at 44 168
\pinlabel $a$ at 232 168
\pinlabel $b^*$ at 364 168
\pinlabel $b$ at  510 168
\pinlabel $a^*c^*$ at  43 24
\pinlabel $a^*a$ at 218  24
\pinlabel $b^*b$ at  371 24
\pinlabel $b^2$ at  505 24
\pinlabel $+$ at  107 24
\pinlabel $+$ at  313 24
\pinlabel $+$ at  438 24
\endlabellist
\includegraphics[scale=0.25]{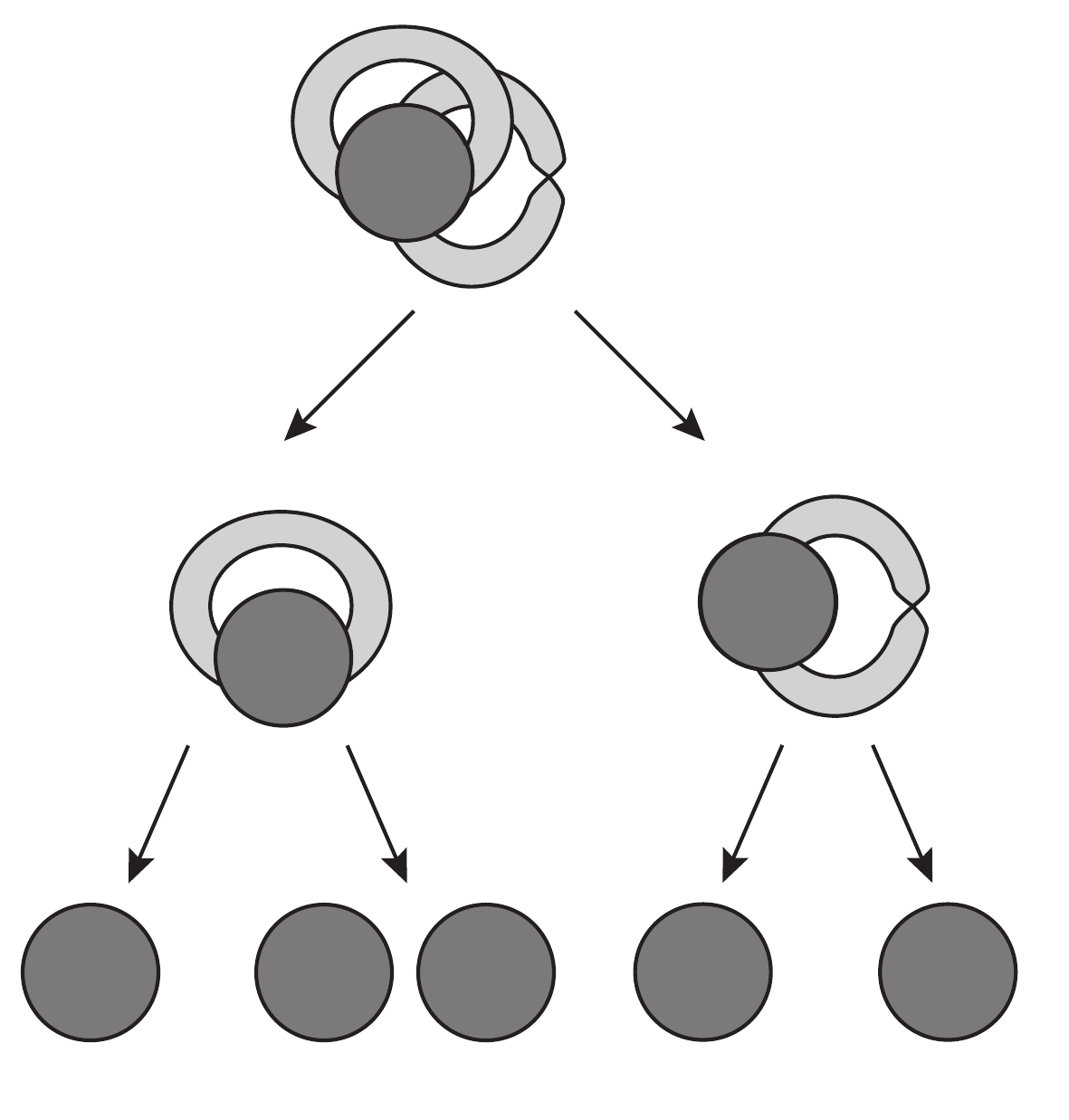}
        \caption{Applying to $f$ then $e$.}
        \label{fwdd}
     \end{subfigure}
     \caption{A computation using the cases in~\eqref{eq:tdd4}.}
\label{fwd}
\end{figure}
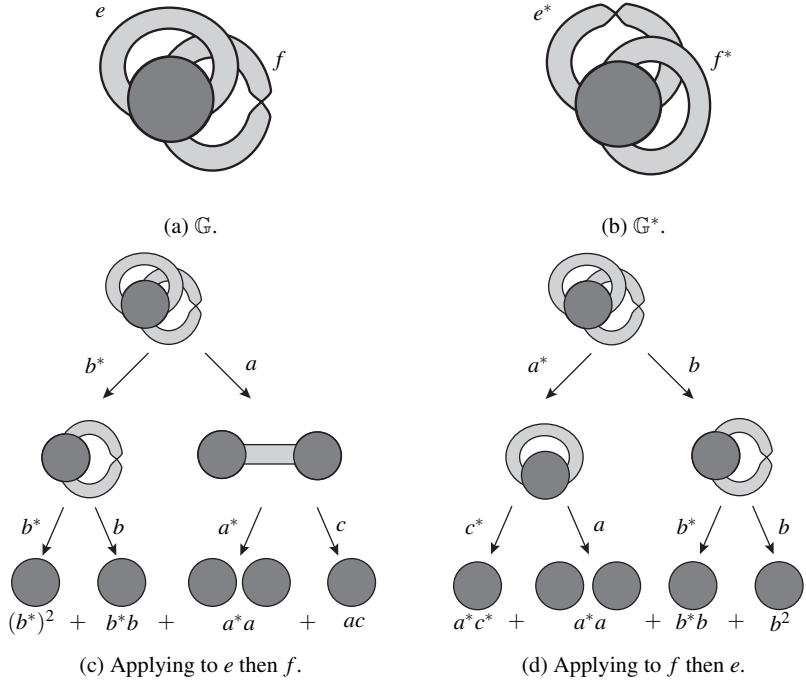

 \begin{definition}[Tutte polynomial --- third definition]\label{def3}
The \emph{Tutte polynomial}, $T(\bG;x,y)$, of  an embedded graph  $\bG$ is defined as
\begin{equation}\label{eq:tdd5}
T(\bG) = \begin{cases}  f(e) \cdot  T(\bG\ba e) +  g(e) \cdot  T(\bG / e)  ,  & \text{for an edge $e$,} 
\\1 & \text{if $\bG$ is edgeless,} \end{cases}
\end{equation}
where
\begin{equation}\label{eq:tdd6}
\begin{array}{l}
f(e) =\begin{cases} x-1 & \text{if $e^*$ is an orientable loop in $\bG^*$,} \\ 
\sqrt{x-1} & \text{if $e^*$ is a nonorientable loop in $\bG^*$,} \\ 
1 & \text{if $e^*$ is a not   loop in $\bG^*$;}  \end{cases}
\\ \text{and}\\
g(e) =\begin{cases} y-1 & \text{if $e$ is an orientable loop in $\bG$,} \\ 
\sqrt{y-1} & \text{if $e$ is a nonorientable loop in $\bG$,} \\ 
1& \text{if $e$ is  not a loop in $\bG$.}  \end{cases}
\end{array}
 \end{equation}
\end{definition}

 \begin{example}\label{examp6}
 Figure~\ref{fig:ztdc} provides an example of a computation of  $T(\bG;x,y)$ using Definition~\ref{eq:tdd5}. There $\bG$ consists of the graph consisting of one vertex and two edges embedded in the torus. For this $T(\bG;x,y) = (y-1)+2(x-1)(y-1)+(x-1)=2xy-x-y$.
 \end{example}
 
 \begin{figure}[ht]
     \centering
        \centering
       \labellist
\small\hair 2pt
\pinlabel $x-1$ at 180 436
\pinlabel $y-1$ at 452 435
\pinlabel $1$ at 61 175
\pinlabel $y-1$ at 234 175
\pinlabel $x-1$ at 386 175
\pinlabel $1$ at  564 175
\pinlabel $(x-1)$ at  52 18
\pinlabel $(x-1)(y-1)$ at  225 18
\pinlabel $(y-1)(x-1)$ at  400 18
\pinlabel $(y-1)$ at  550 18
\pinlabel $+$ at  125 18
\pinlabel $+$ at  318 18
\pinlabel $+$ at  499 18
\endlabellist
\includegraphics[scale=0.5]{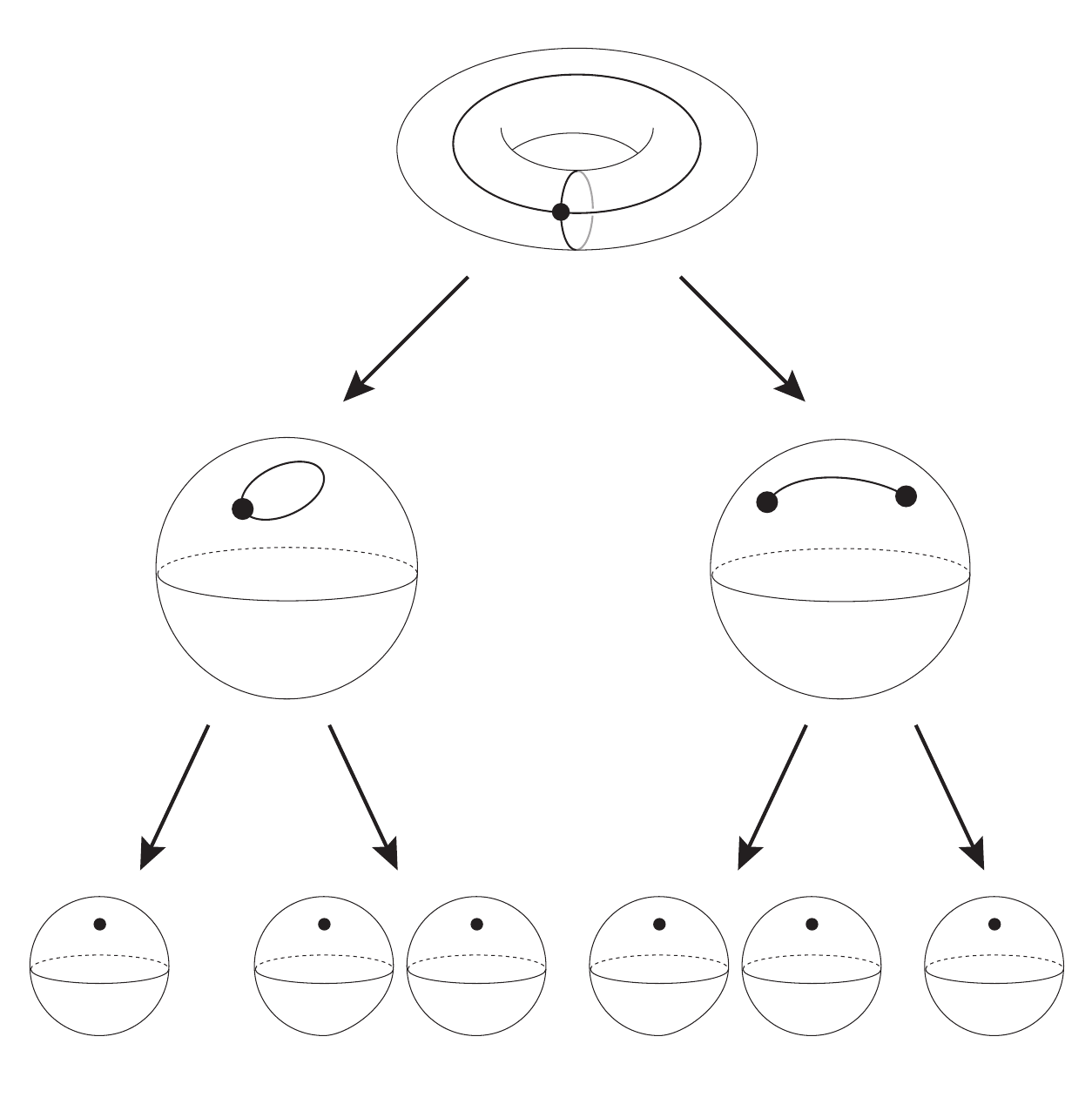}
     \caption{Deletion-contraction computations of $T(\bG;x,y)$.}
\label{fig:ztdc}
\end{figure}

A standard argument shows that the sum in Equation~\eqref{eq:tdef2} satisfies the definition-contraction relations in Definition~\ref{def3}, and consequently that this definition does indeed give a well-defined polynomial (see~\cite[Section~4.1]{HUGGETT_2019} for details). Thus Definitions~\ref{def:ts1},~\ref{def2}, and~\ref{def3} all result in the same polynomial $T(\bG;x,y)$. 

\medskip

The approach in this subsection also leads to a Universality Property from~\cite[Theorem~8]{Krajewski_2018} and~\cite[Theorem~33]{HUGGETT_2019}.

\begin{theorem}[Universality]\label{t:univ}
Let $\mathcal{G}$ be a minor-closed class of embedded graphs or ribbon graphs. Then there is a unique map $U: \mathcal{G}\rightarrow \mathbb{Z}[w, \sqrt{a^*}, \sqrt{c^*},\sqrt{a},\sqrt{c}]$ that satisfies

\[T(\bG) = \begin{cases}  f(e) \cdot  T(\bG\ba e) +  g(e) \cdot  T(\bG / e)  ,  & \text{for an edge $e$,} 
\\w^{v(\bG)} & \text{if $\bG$ is edgeless;} \end{cases}
\]
where 
\[\begin{array}{l}
f(e) =\begin{cases} a^* & \text{if $e^*$ is an orientable loop in $\bG^*$,} \\ 
\sqrt{a^*}\sqrt{c^*} & \text{if $e^*$ is a  nonorientable loop in $\bG^*$,} \\ 
c^* & \text{if $e^*$ is a  not a  loop in $\bG^*$;}  \end{cases}
\\ \text{and}\\
g(e) =\begin{cases} a & \text{if $e$ is an orientable loop in $\bG$,} \\ 
\sqrt{a}\sqrt{c} & \text{if $e$ is a nonorientable loop in $\bG$,} \\ 
c & \text{if $e$ is   not a loop in $\bG$,}  \end{cases}
\end{array}
\]
Moreover,
\[
U(\bG)=  
 w^{v(\bG)-\rho(\bG)}
c^{\rho(\bG)} 
(c^*)^{e(\bG)-\rho(\bG)} 
T
\left(\bG;\frac{wa^*+c}{c},\frac{wa+c^*}{c^*}\right).
\]
\end{theorem}

\begin{example}
If we take $w=v$, $\sqrt{a^*}=\sqrt{c^*}=1$ and $\sqrt{a}=\sqrt{c}=\sqrt{u}$ then the deletion-contraction relations in Theorem~\ref{t:univ} coincides with that for $Z(\bG;u,v)$ in Equation~\eqref{eq:z1}. It follows that
\begin{equation}\label{eq:tz1}
Z(\bG;u,v) = v^{v(\bG)-\rho(\bG)} u^{\rho(\bG)} \, T\Big(\bG; \frac{v+u}{u}, vu+1\Big).  \end{equation}
\end{example}

As a second example of the Universality Theorem we shall prove the duality relation for $T(\bG;x,y)$ discovered independently in~\cite{Ellis_Monaghan_2011} and~\cite{Moffatt_2008} .
\begin{theorem}[Duality]
Let $\bG$ be an embedded graph. Then
\[T(\bG^*;x,y) = T(\bG;y,x). \]
\end{theorem}
\begin{proof}
Define a function $T^*$ from embedded graphs to polynomials in variables $\sqrt{x-1}$ and $\sqrt{y-1}$ by setting $T^*(\bG) =T(\bG^*)$. 
By Definition~\ref{def3}, 
\begin{align*} T^*(\bG;x,y)=T(\bG^*;x,y) &= f(e^*) T(\bG^*\ba e^*;x,y)+g(e^*) T(\bG^* / e^*;x,y)
\\ &= f(e^*) T((\bG/ e)^*;x,y)+g(e^*) T((\bG \ba e)^*;x,y)
\\&= g(e^*) T^*(\bG\ba e;x,y) + f(e^*) T^*(\bG/ e;x,y),
\end{align*}
where we have used that $\bG^*/e^*=(\bG\ba e)^*$ and $\bG^*\ba e^*=(\bG\con e)^*$ (see, for example~\cite[Section~1.9]{zbMATH05569114}, or~\cite[Section~4.2]{Ellis_Monaghan_2013}) for the second equality.
By rewriting   $f(e^*)$ and $g(e^*)$ in terms of $e$ 
we see that   $T^*(\bG;x,y)$ satisfies the deletion-contraction relations for $T(\bG;y,x)$. 
(For example, $f(e^*)=x-1$ if $(e^*)^*$ is an orientable loop in $(\bG^*)^*$, which, as duality is involutory, happens  if $e$ is an orientable loop in $\bG$.)
It follows from the Universality Theorem that $T^*(\bG;x,y)=T(\bG;y,x)$ and so $T(\bG^*;x,y) = T(\bG;y,x)$.\qed
\end{proof}
 It is worth highlighting that $T(\bG^*;x,y) = T(\bG;y,x) $  extends the classical result  that for planar graphs $T(G;x,y)=T(G^*;y,x)$ to non-planar graphs and non-plane duals.

\medskip

We say that a ribbon graph $\bG$ is the  \emph{join} of ribbon graphs $\bG'$ and $\bG''$, written $\bG' \vee \bG''$, if $\bG$ can be obtained by identifying an arc on the boundary of a vertex of  $\bG'$ with an arc on the boundary of a vertex of $\bG''$. (The two arcs cannot intersect edges.) The two  vertices with identified arcs make a single vertex of $\bG$. (See, for example, \cite{Ellis_Monaghan_2013,Moffatt_2012} for an elaboration of this operation.) A description of this process in terms of embedded graphs can be found in the first paragraph of \cite[Section~3.5.2]{zbMATH04006288}.

By applying deletion and contraction to all of the edges in $\bH$ before the edges in $\bG$ we have that
\begin{equation}\label{eq:join1}
T(\bG\sqcup \bH;x,y) =  T(\bG;x,y)\cdot T(\bH;x,y) = T(\bG\vee \bH;x,y),
\end{equation}
and, by making use of Equation~\eqref{eq:tz1}, that 
\begin{equation}\label{eq:join2}
Z(\bG\sqcup \bH;u,v) =  Z(\bG;u,v)\cdot Z(\bH;u,v) = v\cdot Z(\bG\vee \bH;u,v).
\end{equation}

\subsection{Quasi-trees}\label{sec:qt}
In this section we show that $T(\bG;x,y)$ can be written as a sum over spanning quasi-trees, a result that is analogous the fact that the classical Tutte polynomial of a connected graph can be written as a sum over spanning trees. 
An activities expansion for the Bollob\'as--Riordan polynomial of orientable ribbon graphs was given  in~\cite{zbMATH05960754}. This was quickly extended to non-orientable ribbon graphs in~\cite{Vignes_Tourneret_2011} and, independently, in an unpublished manuscript~\cite{dewey}. 
An extension to the  Krushkal polynomial followed in~\cite{Butler_2018}. Each of these expansions expresses the graph polynomial as a sum over quasi-trees, but they all include a Tutte polynomial of an associated graph as a summand. Most recently,~\cite{Morse_2019} gave a spanning tree expansion for the 2-variable Bollob\'as--Riordan polynomial  of a delta-matroid which specialised to give one for the 2-variable Bollob\'as--Riordan polynomial of a ribbon graph. The approach we take here is a special case of~\cite[Section~5]{HUGGETT_2019} which gives expressions for the Tutte polynomials of several types of topological graph.

\medskip

Equations~\eqref{eq:tdc} and~\eqref{eq:tsum} gave the two most common definitions of the Tutte polynomial of a graph. However there is a third standard definition, often called the \emph{spanning tree expansion} or \emph{activities expansion}, that expresses the Tutte polynomial of a  graph as a sum over spanning trees. 
For completeness we include a description of the definition; however, a reader unfamiliar with it may safely skip the details in the next    paragraph.

Let $G$ be a connected graph and fix a linear order of its edges. 
Suppose that $T$ is a  spanning tree of $G$ and let $e$ be an edge of $G$.
If $e\notin T$, then the graph $T\cup e$ contains a unique cycle, and we say that $e$ is \emph{externally active} with respect to $T$ if it is the smallest edge in this cycle.
If $e\in F$, then we say that  $e$ is \emph{internally active} if $e$ is the smallest edge of $G$ that can be added to $T\ba e$ to recover a spanning tree of $G$.
Then the  \emph{spanning tree expansion} of the Tutte polynomial is
\begin{equation}\label{eq:tact}
T(G;x,y) = \sum_{T\in \mathcal{T}} x^{IA(T)}y^{EA(T)},
\end{equation}
where $\mathcal{T}$ is the set of spanning trees of $G$, and where $\mathrm{IA}(T)$ (respectively $\mathrm{EA}(T)$) denotes the number internally active (respectively, externally active) edges of $G$ with respect to $T$ and the linear ordering of $E$. (Additional details and examples can be found in, for example, \cite[Section~2.2.3]{Ellis_Monaghan_2022_C}).
We can obtain a similar expression for the Tutte polynomial of an embedded graph, but instead of summing over spanning trees we sum over spanning quasi-trees as follows.

 \medskip

For convenience, in the rest of  subsection we shall work in the language of ribbon graphs and restrict our discussions to connected ribbon graphs.

We begin by rewriting Equation~\eqref{eq:zdc}.
An edge in a ribbon graph is a \emph{bridge} if deleting it increases the number of components of the ribbon graph. 
A loop $e$  in a ribbon graph is said to be \emph{interlaced} with a cycle $C$ if when travelling around the boundary of the vertex incident to $e$ we see edges in the cyclic order $e \, c_1\, e \,c_2$ where $c_1$ and $c_2$ are edges of $C$.
A loop is \emph{trivial} if it is not interlaced with any cycle. 
With this terminology we can use  Equation~\eqref{eq:join2} to rewrite Equation~\eqref{eq:zdc} as
\begin{equation}\label{eq:zdc2}
Z(\bG;u,v) =\begin{cases}  
  (1+uv)  Z(\bG\ba e;u,v )  &  \text{if $e$ is a trivial orientable loop,}\\
  (u+v)\, Z(\bG\con e;u,v  ) &  \text{if $e$ is a bridge,}\\
 Z(\bG\ba e;u,v ) + u\, Z(\bG\con e;u,v  ) &  \text{if $e$ is not a bridge}
 \\
&  \qquad\text{nor a trivial orientable loop,}\\
    v^{|V|} &  \text{if }E(G)=\emptyset.
 \end{cases} \end{equation}
Notice that if $\bG$ connected every ribbon graph arising in a computation of $Z(\bG;u,v)$ using~\eqref{eq:zdc2} is also connected.

We consider \emph{resolution trees} for the computation of $Z(\bG;u,v)$ via the deletion-contraction relations in~\eqref{eq:zdc2}. An example of one  is given in Figure~\ref{gsdh}.  The figure may suffice as an alternative to the formal definition that follows.
Given a ribbon graph $\bG$ choose a linear order on its edge set.
The resolution tree associated with this linear order is a rooted tree  of height $e(\bG)$ whose nodes are the minors of $\bG$ that appear in the computation of $Z(\bG;u,v)$ undertaken with respect to the linear order. The node of height $e(\bG)$ is  $\bG$. 
Each node of height $i>0$ gives rise to one or two nodes of height $i-1$ as follows. 
Suppose that $\bH$ is a node of height $i$ and $e$ is the $i$-th edge.
If $e$  is a bridge, then insert a node $\bH/e$ at height $i-1$ and a tree-edge labelled $(u+v)$ is added between the nodes $\bH$ and $\bH/e$. 
If $e$  is  a trivial orientable  loop, then insert a node $\bH\ba e$ at height $i-1$ and a tree-edge labelled $(1+uv)$ is added between the nodes $\bH$ and $\bH\ba e$. 
Otherwise insert both $\bH\ba e$ and $\bH/ e$ as nodes of height $i-1$, and add a tree-edge labelled $1$  between the nodes $\bH$ and $\bH\ba e$, and a tree-edge labelled $u$  between the nodes $\bH$ and $\bH/ e$. 
The \emph{root} is the node at height $e(\bG)$ corresponding to the original graph. 
A \emph{branch} is a path from the root to a leaf. 
The leaves are of height 0, and each is a ribbon graph consisting of a single vertex.

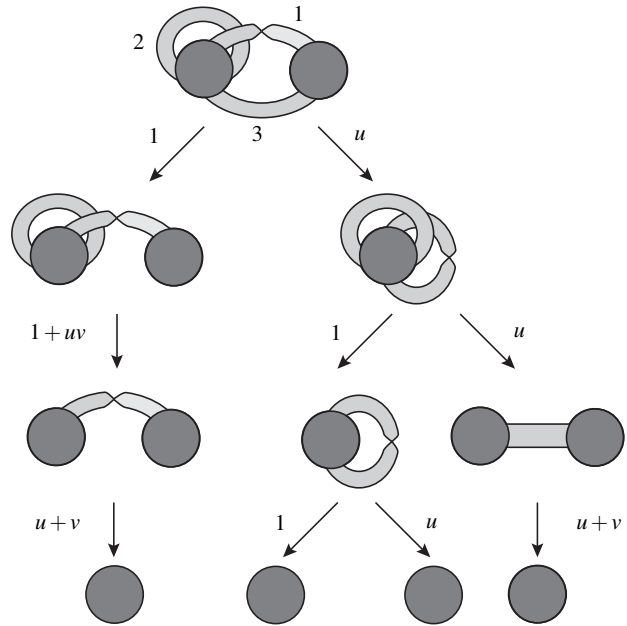
\begin{figure}[ht]
\centering
\labellist
\small\hair 2pt
\pinlabel $3$ at 311 620
\pinlabel $2$ at 162 735
\pinlabel $1$ at 363 770
\pinlabel $1$ at 181 617
\pinlabel $u$ at 440 617
\pinlabel $1+uv$ at  60 370
\pinlabel $1$ at  410 370
\pinlabel $u$ at  640 370
\pinlabel $u+v$ at  60 130
\pinlabel $1$ at  340 130
\pinlabel $u$ at  530 130
\pinlabel $u+v$ at  740 130
\endlabellist
\includegraphics[scale=0.3]{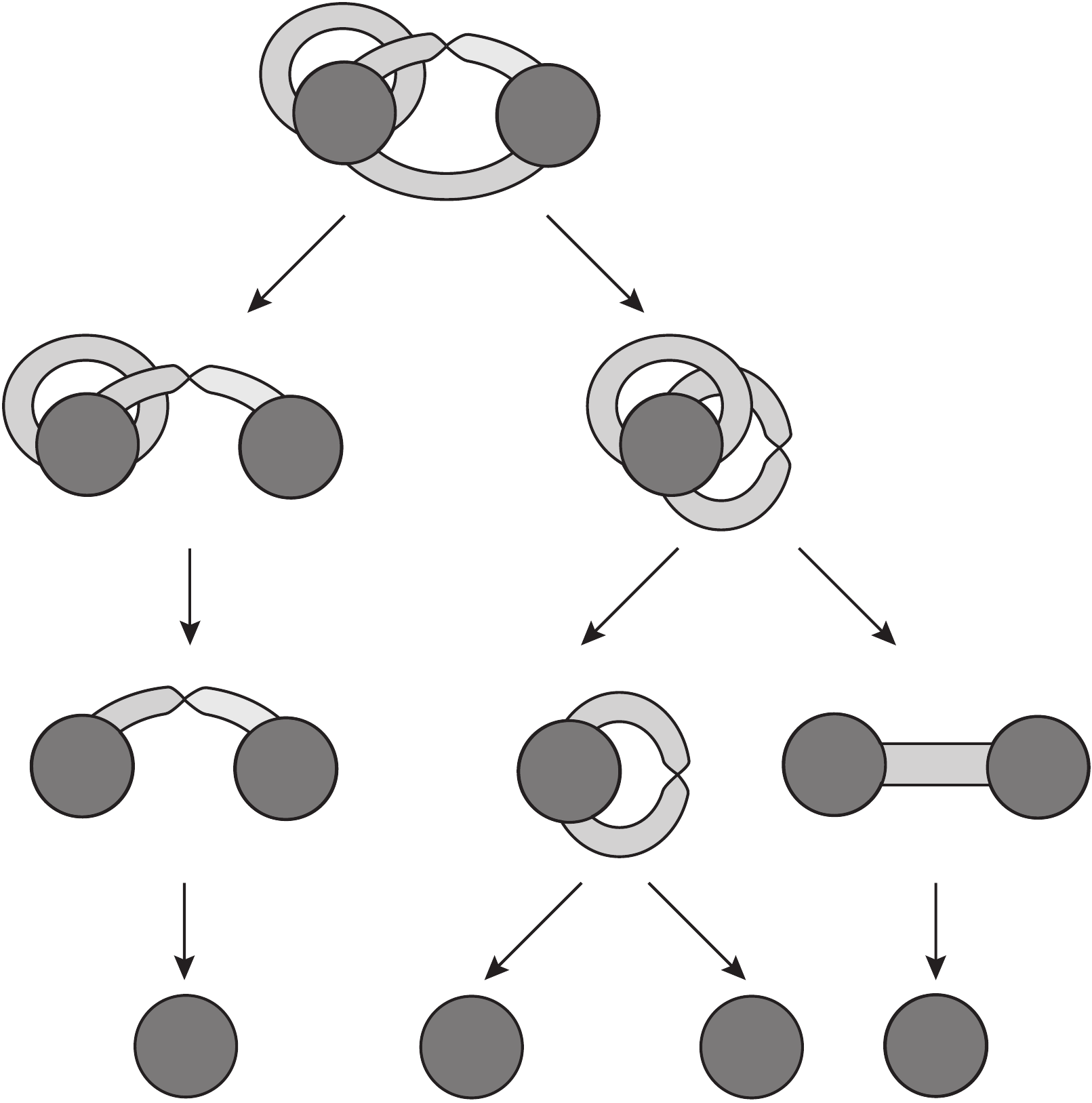}
\caption{A resolution tree.}
\label{gsdh}
\end{figure}

Observe that $Z(G;u,v)$ is obtained from a resolution tree by, for each branch, taking the product of the tree-edge labels on that branch and multiplying this by $v$ (the contribution of the isolated vertex at the leaf), then summing together all the resulting terms.
For example,  Figure~\ref{gsdh} gives
\begin{align*}
Z(G;u,v) &= v(1\cdot(1+uv)\cdot(u+v)) +v(u\cdot 1 \cdot 1)+v(u\cdot 1 \cdot u)+v(u\cdot u \cdot (u+v))
\\&=u^{3} v +2 u^{2} v^{2}+u \,v^{3}+ u^{2}\,v+2 u v +v^{2}
.\end{align*}

\medskip

The branches of the resolution tree correspond to spanning quasi-trees.
A ribbon graph is a \emph{quasi-tree} if it has exactly one boundary component. A ribbon subgraph of $\bG$ is \emph{spanning} if it contains each vertex of $\bG$. Note that a genus 0 quasi-tree is a tree.

\begin{lemma}\label{zxzc}
If $\bG$ is a connected  ribbon graph then the set of spanning quasi-trees of $\bG$ is in one-one correspondence with the set of leaves of the resolution tree of $Z(\bG;u,v)$ computed from Equation~\eqref{eq:zdc2}. Furthermore, the correspondence is given by deleting the set of edges of $\bG$ that are deleted in the branch that terminates in the node. 
\end{lemma}
\begin{proof}[Sketch]
If you start with a connected ribbon graph the only way you can create an additional component is by deleting a bridge or contracting a trivial orientable loop. Thus the branches in a resolution tree correspond to every possible way to delete and contract edges without creating additional components. Next, since the order of deletion and contraction of edges does not matter, the ribbon graph at leaf is a vertex obtained as $\bG\ba D /C$, for some sets of edges $C$ and $D=E-C$. Since contraction does not change the number of boundary components, it follows that $\bG\ba D$ has exactly one boundary component and is hence a spanning quasi-tree of $\bG$. The 1-1 correspondence follows. (A more detailed proof can be found just after~\cite[Lemma~46]{HUGGETT_2019}.)\qed
\end{proof}

In light of Lemma~\ref{zxzc} we see that if we can read off the product of the labels on a branch directly from its corresponding spanning quasi-tree then we can express $Z(G;u,v)$ as a sum over spanning quasi-trees. 
The edges of the given spanning quasi-tree correspond to those edges contracted along the branch, and edges outside the quasi-tree are deleted. Thus to establish how much the branch contributes, we just need to know how many edges we see as a bridge and how many edges we see as a trivial orientable loop when we delete and contract edges as specified by the branch.

Recall that the edges of $\bG$ are ordered. Label the edges of $\bG$ according to this ordering.
Consider a spanning quasi-tree $\bT$ as sitting inside $\bG$. 
Arbitrarily orient the boundary of each edge (so there is a preferred direction of travel around an edge.)
Next travel round the boundary of $\bT$ and read off the edge-names as we encounter them. This gives a double-occurrence word $\omega$ in the edge-names.  
As we travel round the boundary of $\bT$ we meet each edge twice. Furthermore, we meet it in a direction that is consistent with the edge-boundary orientation or inconsistent with the edge-boundary orientation.
If we  meet the edge $i$ once consistently and once inconsistently then put a bar over the two occurrences of $i$ in $\omega$. 

\begin{example}\label{eq:act1}
Consider the ribbon graph in Figure~\ref{fact1}. The edge set $\{3,5,6,7\}$ defines one of its spanning quasi-trees. This is shown  by the darker edges in  Figure~\ref{fact2}. The arrows indicate the edge orientations. We have $\omega=135\bar{7}\bar{8}\bar{6}\bar{4}\bar{7}\bar{8}\bar{6}53\bar{4}\bar{2}1\bar{2}$ (up to cyclic permutations and reversals).
\end{example}

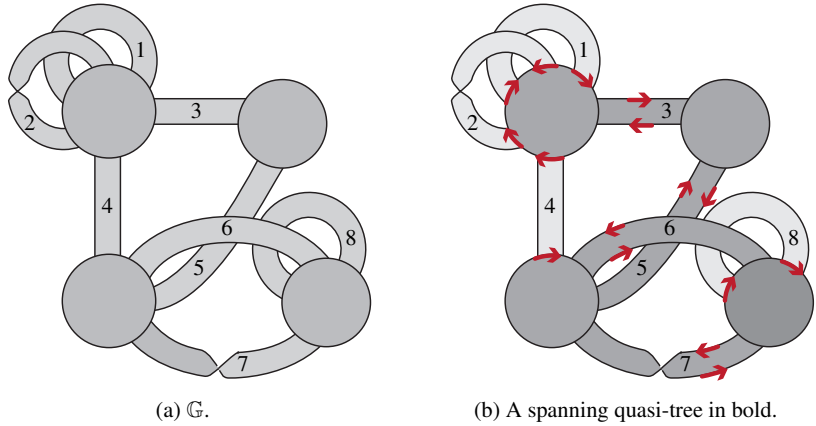
\begin{figure}
     \centering

        \hfill
        \begin{subfigure}[c]{0.45\textwidth}
        \centering
        \labellist
\small\hair 2pt
\pinlabel {$1$}  at  106 228
\pinlabel {$2$}   at    30 178
\pinlabel {$3$}  at    144 186
\pinlabel {$4$}   at   83 121
\pinlabel {$5$}   at     146 79
\pinlabel {$6$}  at   165 106 
\pinlabel {$7$}   at    176 13
\pinlabel {$8$}   at  250 99
\endlabellist
\includegraphics[scale=.55]{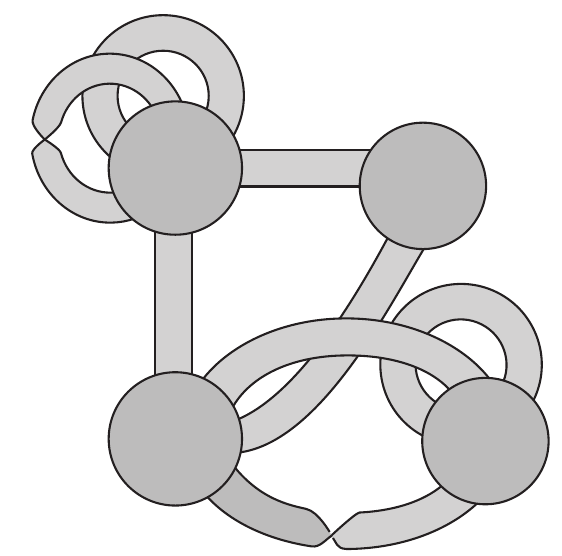}
        \caption{$\bG$.}
        \label{fact1}
     \end{subfigure}
        \hfill
        \begin{subfigure}[c]{0.45\textwidth}
        \centering
        \labellist
 \small\hair 2pt
\pinlabel {$1$}  at  106 228
\pinlabel {$2$}   at    30 178
\pinlabel {$3$}  at    163 186
\pinlabel {$4$}   at   83 121
\pinlabel {$5$}   at     146 79
\pinlabel {$6$}  at   165 106 
\pinlabel {$7$}   at    176 13
\pinlabel {$8$}   at  250 99
\endlabellist
\includegraphics[scale=.55]{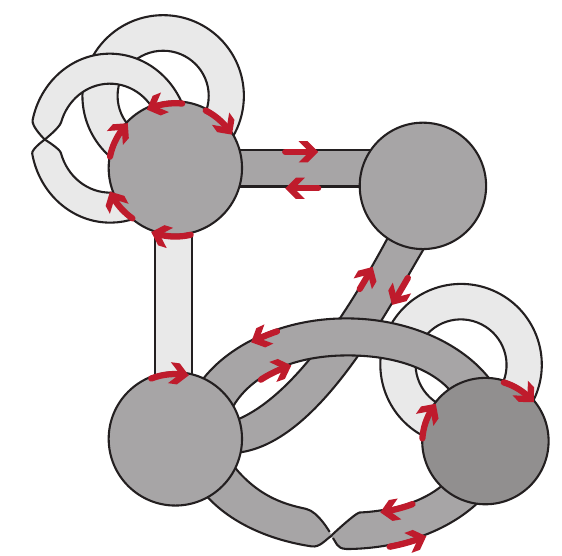}
        \caption{A spanning quasi-tree in bold.}
        \label{fact2}
     \end{subfigure}
     \caption{Computing $\omega$ from a spanning quasi-tree.}
\label{fact}
\end{figure}

Two labels $i$ and $j$ are \emph{interlaced} in $\omega$ if we see they appear in the cyclic order $\cdots i\cdots j\cdots i\cdots j$, with or without bars.
We say that an edge $i$ is \emph{live} in $\omega$ if it is not interlaced with a lower order edge. 

For example, with $\omega$ as in Example~\ref{eq:act1}, edges 1 and 3 are live, while the remaining  edges are not.

It is not too hard to see that:
\begin{itemize}
\item  $i$ is in $\bT$ and live if and only if it is a bridge or nonorientable trivial loop in the height $i$ node of the branch of the resolution tree corresponding to $\bT$.
\begin{itemize}
\item  If in addition $i$ appears in $\omega$ with bars then it is a nonorientable trivial loop. 
\item If in addition $i$ appears in $\omega$ without bars then it is a bridge. 
\end{itemize}
\item $i$ is not in $\bT$ and live  if and only if it is a trivial loop in the height $i$ node of the branch of the resolution tree corresponding to $\bT$.
\begin{itemize}
\item  If in addition $i$ appears in $\omega$ with bars then it is a nonorientable trivial loop. 
\item If in addition $i$ appears in $\omega$ without bars then it is a orientable trivial loop. 
\end{itemize}
\end{itemize}

Thus if we let $ILO(\bT)$ denote the number of live edges in $\bT$ that appear in $\omega$ without bars, and we let 
$ELO(\bT)$ denote the number of live edges that are not in $\bT$ and appear in $\omega$ without bars, then 
\begin{equation}\label{eq:qtree}
Z(\bG;u,v) = v\, \sum_{\bT\in \mathcal{Q}(\bG)}  u^{e(\bT)} (1+u^{-1}v)^{ILO(\bT)}(1+uv)^{ELO(\bT)}  , 
\end{equation}
where $\mathcal{Q}(\bG)$ is the set of spanning quasi-trees in $\bG$.

\begin{example}
Let $\bG$ be the ribbon graph in Figure~\ref{gsdh} with edge set $\{1,2,3\}$. This has four spanning quasi-trees given by the edge sets $\{1\}$, $\{3\}$, $\{1,3\}$, $\{1,2,3\}$. 
For $\{1\}$, we have $\omega=12\bar{3}21\bar{3}$ and edges $1$ and $2$ are live. Thus this spanning tree contributes  $u (1+u^{-1}v)(1+uv)$ to the sum in~\eqref{eq:qtree}.
For $\{3\}$, we have $\omega=\bar{1}32\bar{1}23$ and only the edge $1$ is live. Thus this spanning tree contributes  $u$ to the sum in~\eqref{eq:qtree}.
For $\{1,3\}$, we have $\omega=\bar{1}\bar{2}\bar{3}\bar{1}\bar{2}\bar{3}$ and only the edge  $1$ is live. Thus this spanning tree contributes  $u^2$ to the sum in~\eqref{eq:qtree}.
For $\{1,2,3\}$, we have $\omega=1\bar{2}1{3}\bar{2}{3}$ and only the edge  $1$ is live. Thus this spanning tree contributes  $u^3(1+u^{-1}v)$ to the sum in~\eqref{eq:qtree}.
Thus 
\begin{align*}
Z(\bG;u,v) &= v(u (1+u^{-1}v)(1+uv)+u+u^2+ u^3(1+u^{-1}v) )
\\&=u^{3} v +2 u^{2} v^{2}+u \,v^{3}+u^{2}\,v+2 u v +v^{2}
.\end{align*}

\end{example}

\subsection{Further properties}\label{sec:further}

We conclude by listing some key properties of the Tutte polynomial of embedded graphs. 

\begin{enumerate}

\item $T(\bG;x,y)$ is determined by the delta-matroid of $\bG$. See~\cite{CMNR-JCT}. This is analogous to the well-known fact that the classical Tutte polynomial of a graph is determined by the cycle matroid of the graph. 

\item If $\bG$ is a plane embedding of $G$, then $T(\bG;x,y)=T(G;x,y)$. See~\cite{BR01,BR02}.

\item  If $\bG$ is an embedded graph and $H$ a graph such that  $\R(\bG;x,y)=T(H;x,y)$, then $\bG$ is plane. See~\cite[Corollary~3.11]{Ellis_Monaghan_2022_B}. 

\item Let $\bG$ be an orientable embedded graph. Then $T(\bG;x,y)$ is irreducible over $\mathbb{Z}[x,y]$ (or $\mathbb{C}[x,y]$) if and only if $\bG $ is non-separable. See~\cite[Theorem~1.1]{Ellis_Monaghan_2022_B}.

\item $T(\bG\sqcup \bH;x,y) =  T(\bG;x,y)\cdot T(\bH;x,y) = T(\bG\vee \bH;x,y)$. (See~\cite{BR01,BR02}.)
\item $T(\bG^*;x,y)=T(\bG;y,x)$. See~\cite{Ellis_Monaghan_2011,Moffatt_2008}. This expression was extended to a partial duality formula for a multivariate version of the dichromatic polynomial in~\cite{zbMATH05569114,Moffatt_2010_PDBR,Vignes_Tourneret_2009}. 

\item An analogue of the classical result that two graphs related by Whitney flips have the same Tutte polynomial holds (see, for example, \cite[Theorem~6.14]{zbMATH07680519} for details on this classical result).  
If two embedded graphs are related by vertex joins, vertex cuts, or mutation, as defined in~\cite{zbMATH07336895}, then their Tutte polynomials are identical. See~\cite[Corollary~3]{zbMATH07336895} for details.

\item A expression for the Tutte polynomial of the $k$-sum of two embedded graphs in terms of polynomials associated with its two $k$-summands was given in~\cite[Theorem~1]{HMsplit}. This is an analogue of Negami's splitting formula for the classical Tutte polynomial of a graph~\cite{Negami_1987}. 

\item An expression for $T(\bG\otimes \bH;x,y)$  in terms of $T(\bG;x,y)$,  $T(\bH\ba e;x,y)$ and $T(\bH/e;x,y)$ was given in~\cite[Corollary~4.4]{zbMATH05999807} and~\cite[Corollary~6.3]{ELLIS_MONAGHAN_2014}. 
This is an  analogue of Brylawski's tensor product formula for the classical Tutte polynomial of a graph from~\cite{Brylawski_2010}.

\item A convolution product formula, which for orientable embedded graphs states that 
\[ T(\bG;x,y) =\sum_{A\subseteq E} T(\bG \ba (E-A) ; 0,y ) \cdot  T(\bG / A ; x,0 ) , \]
was given in~\cite[Theorem~10]{Krajewski_2018}.

\item
$T(\bG;x,y)$ determines the following (see~\cite[Corollary~3.3]{Ellis_Monaghan_2022_B}).
\begin{enumerate}
\item The number of edges of $\bG$.
\item When $\bG$ is connected, the number of spanning quasi-trees of $\bG$ with a given number of edges.
\item The ranks $r(\bG)$ and $r(\bG^*)$.
\item The Euler genus $\gamma(\bG)$ 
of $\bG$.
\item  Whether or not $\bG$ is plane.
\item Whether or not $\bG$ is orientable.
\item The  genus $g(\bG)$.
\end{enumerate} 

\item 
For a ribbon graph $\bG$ with underlying graph $G$,  let $H$ be the underlying graph of $\bG^*$. Each of the following hold. (See~\cite[Theorem 6.6]{CMNR-JCT}).
\begin{enumerate}
    \item $T(\bG;0,0)$ is $1$  if $\bG$ is trivial and 
    $0$ otherwise.
    \item $T(\bG;1,1)=0$ unless $\bG$ is plane, in which case it equals the number of maximal spanning forests of $G$ (spanning trees if $G$ is connected).
    \item $T(\bG;1,2)$ is the number of spanning forests in $H$.
    \item $T(\bG;2,1)$ is the number of spanning forests in $G$. 
    \item $T(\bG;2,2)=2^{|E|}$.
\end{enumerate}

\item An expression for  $T(\bG;3,3)$ in terms of claw coverings  was given in ~\cite[Theorem~9]{KP03}, and in terms of Tetromino tilings in~\cite[Corollary~11]{KP03}.

\item For $k\in \mathbb{N}$,   $T(\bG;k+1,k+1)$ can be expressed in terms of edge colourings in medial graphs. See~\cite[Theorem~5]{KP03}. This was extended in~\cite[Evaluation 5]{ELLIS_MONAGHAN_2014} to an interpretation of $T\Big(\bG;  \frac{k}{b}+1, bk+1\Big) $.

\item $|T(\bG;-1,-1)|= 2^{c(\bG)-k(\bG)+\gamma(\bG)/2}$, where $c(\bG)$ is the number of all crossing components in the medial graph of $\bG$. See~\cite[Theorem~3.1]{zbMATH06873103}.

\item The \emph{beta invariant}, $\beta(\bG)$, of a ribbon graph  $\bG$ is the coefficient of $x$ in $T(\bG;x,y)$.
\begin{enumerate}
\item 
Let $\bG$ be an orientable ribbon graph with at least two edges. Then
$\beta(\bG) \ne 0$ if and only if $\bG$ is non-separable.
Moreover if $\bG$ is non-separable, then the sign of $\beta(\bG)$ is the same as that of $(-1)^{\gamma(\bG)/2}$.  See \cite[Proposition~4.3]{Ellis_Monaghan_2022_B}.

\item 
Let $\bG$ be an orientable ribbon graph with at least two edges.  
Then the following are equivalent.
\begin{itemize}
    \item After removing isolated vertices, $\bG$ is series-parallel.
    \item $\beta(\bG)=(-1)^{\gamma(\bG)/2}$.
\end{itemize}
See~\cite[Theorem~4.9]{Ellis_Monaghan_2022_B}.
\end{enumerate}

\item
Let $\bG=(V,E)$ be an orientable ribbon graph and $T(\bG;x,y)=\sum_{i,j\geq 0} b_{i,j}x^iy^j$. Then, 
for all $k$ with $0\leq k < e(\bG)$,
    \[ \sum_{i=0}^k \sum_{j=0}^{k-i} (-1)^j \binom{k-i}j b_{i,j}=0,\]
and
\[ \sum_{i=0}^{e(\bG)} \sum_{j=0}^{e(\bG)-i} (-1)^j \binom{e(\bG)-i}j b_{i,j}=(-1)^{e(\bG)-\rho(E)}.\]
See~\cite[Theorem~5.1]{Ellis_Monaghan_2022_B}. These identities are analogous to Brylawski's coefficient identities for the classical Tutte polynomial of a graph from~\cite{zbMATH03355044}.

\item $T(\bG;x,y)$ can be recovered from both the topological transition polynomial (see~\cite[Theorem 4.2]{Ellis_Monaghan_2011}) and the Krushkal polynomial of $\bG$ (see~\cite[Lemma~4.1]{KRUSHKAL_2010} and~\cite[Theorem 5.1]{Butler_2018}). 

\item There are several connections between $T(\bG;x,y)$ and knot polynomials, as follows. (See~\cite{Huggett_2022} for an overview.)
\begin{enumerate}
\item $T(\bG;x,y)$ can be used to determine the Jones polynomial of the following types of links: 
all knots and links~\cite{zbMATH05249594};
checkerboard colourable virtual links~\cite{zbMATH05251649},
all virtual links using the dichromatic polynomial~\cite{zbMATH05507913}, 
links in real projective space~\cite{zbMATH06526330}.
\item $T(\bG;x,y)$ is determined by the \uppercase{homfly-pt} polynomial of an associated link in a thickened surface. See~\cite{Moffatt_2008}.
\item A categorification for $T(\bG;x,y)$ was given in~ \cite{zbMATH05235138}.
\end{enumerate}
\end{enumerate}

\bibliographystyle{abbrv}
\bibliography{matrix}

\end{document}